\documentclass[reqno,draft]{amsart}

\usepackage{amssymb,amsmath,amsthm,times,algorithm2e}

\renewcommand{\frak}{\mathfrak}
\newcommand{\bbold}{\mathbb}
\newcommand{\rom}{\textup}
\newcommand{\cal}{\mathcal}

\renewcommand\emptyset{\varnothing}

\def \R {{\bbold R}}
\def \Q {{\bbold Q}}

\def \N {{\bbold N}}

\def \<{\langle}
\def \>{\rangle}

\def \hom {{\operatorname{h}}}

\newcommand{\abs}[1]{\lvert#1\rvert}

\renewcommand\leq{\leqslant}
\renewcommand\geq{\geqslant}

\def \supp {\operatorname{supp}}
\def \lc{\operatorname{lc}}
\def \lm{\operatorname{lm}}
\def \nf{\operatorname{nf}}
\def \gr{\operatorname{gr}}
\def \lcm{\operatorname{lcm}}
\def \Syz{\operatorname{Syz}}
\def \tr{{\operatorname{tr}}}
\def \op{{\operatorname{op}}}
\def \env{{\operatorname{env}}}
\def \lex{{\operatorname{lex}}}
\def \dlex{{\operatorname{dlex}}}
\def \wt{\operatorname{wt}}
\def \reg{\operatorname{reg}}

\newcommand{\red}[1]{\underset{#1}{\longrightarrow}}

\newtheorem{theorem}{Theorem}[section]
\newtheorem{lemma}[theorem]{Lemma}
\newtheorem{prop}[theorem]{Proposition}
\newtheorem{cor}[theorem]{Corollary}

\theoremstyle{definition}
\newtheorem{definition}[theorem]{Definition}

\theoremstyle{remark}
\newtheorem{example}[theorem]{Example}
\newtheorem{examples}[theorem]{Examples}

\newtheorem{remark}[theorem]{Remark}

\numberwithin{equation}{section}

\begin{document}

\title{Degree Bounds for Gr\"obner Bases in Algebras of Solvable Type}

\author{Matthias Aschenbrenner} \email{matthias@math.ucla.edu}
\urladdr{http://www.math.ucla.edu/$\sim$matthias}

\address{Department of Mathematics \\
University of California, Los Angeles \\
Box 951555 \\
Los Angeles, CA 90095-1555 \\
U.S.A.}

\author{Anton Leykin} \email{leykin@math.uic.edu}
\urladdr{http://www.math.uic.edu/$\sim$leykin}

\address{Department of Mathematics, Statistics, and Computer Science\\
University of Illinois at Chicago\\
851 S. Morgan Street (M/C 249)\\
Chicago, IL 60607-7045 }

\date{November 18, 2008}

\thanks{The first author was partially supported by a grant from the National Science Foundation.}

\subjclass[2000]{Primary 13P10, 13N10; Secondary 68Q40, 16S32}

\begin{abstract}
We establish doubly-exponential degree bounds for Gr\"obner bases in certain
algebras of solvable type  over a field (as introduced by Kandri-Rody and Weispfenning).
The class of algebras considered here includes
commutative polynomial rings, Weyl
algebras, and universal enveloping algebras of finite-dimensional Lie algebras.
For the computation of these bounds, we adapt a method due to Dub\'e based
on a generalization of Stanley decompositions.
Our bounds yield doubly-exponential degree bounds for ideal membership and
syzygies, generalizing the classical results of Hermann and Seidenberg
(in the commutative case) and Grigoriev (in the case of Weyl algebras).
\end{abstract}

\keywords{Gr\"obner bases,  Weyl algebras, degree bounds, Stanley decompositions}

\maketitle


\section*{Introduction}

The algorithmic aspects of Weyl algebras were first explored by Castro \cite{Castro}, Galligo \cite{Galligo}, Takayama \cite{Takayama} and others in the mid-1980s. In particular, they laid out a theory of Gr\"obner bases in this slightly non-commutative setting. Since then, Gr\"obner bases in Weyl algebras have been widely used for practical computations in algorithmic $D$-module theory as promoted in \cite{SST}. (Some authors \cite{Chistov-Grigoriev-2007} prefer the term ``Janet basis'' in this context, due to the pioneering work  on linear differential operators by Janet \cite{Janet} in the 1920s.)
In the early 1990s, Kandri-Rody and Weispfenning \cite{KR-W}, by isolating the features of Weyl algebras which permit Gr\"obner basis theory to work, extended this theory to a larger class of non-commutative algebras, which they termed {\it algebras of solvable type}\/  over a given coefficient field $K$. This class of algebras includes the universal enveloping algebras of finite-dimensional Lie algebras over $K$, by a theorem attributed to Poincar\'e, Birkhoff and Witt. (For this reason, algebras of solvable type are sometimes called {\it PBW-algebras}; see, e.g., \cite{BGV, Roman-Roman}. Another designation in use is {\it polynomial rings of solvable type.}\/) Working implementations of these algorithms exist and are in widespread use; see \cite[Section~2.6]{CAH} and \cite{LS}. Similar extensions of Gr\"obner basis theory to non-commutative algebras were studied by Apel \cite{Apel} and Mora \cite{Mora}. See Sections~\ref{Preliminaries on Algebras over Fields} and \ref{Grobner bases in Algebras of Solvable Type} below for a recapitulation of the basic definitions, and \cite{BGV} for a comprehensive introduction to this circle of ideas.

In this paper we are interested in degree bounds for left Gr\"obner bases in algebras of solvable type. It follows trivially from the case of commutative polynomials (as treated in \cite{Huynh}) and Section~\ref{Degree bounds for normal forms} below that  the degrees of the elements of the reduced Gr\"obner basis of a left ideal $I$ in an algebra of solvable type may depend doubly-exponentially on the maximum of the degrees of given generating elements of $I$. In view of the popularity of this kind of non-commutative Gr\"obner basis theory, it is surprising that little seems to be known about {\it upper degree bounds}\/ for Gr\"obner bases (and, by extension, about the worst-case complexity of Buchberger's algorithm) in this setting. Perhaps it was believed that the upper degree bound for one-sided Gr\"obner bases, at least  in the context of Weyl algebras, also follows from the commutative polynomial case by passing to the associated graded algebra for a certain filtration (which turns out to be nothing but a commutative polynomial ring over the given coefficient field). If true, the problem would have boiled down to the doubly-exponential degree bounds for Gr\"obner bases in commutative polynomial rings over fields found in the 1980s (see, e.g., \cite{Moller-Mora-1984}). However, we would like to emphasize that we {\it could not find}\/ and we {\it do not believe there exists}\/ a simple way to establish such a degree bound by reducing the question to commutative algebra. (See Section~\ref{Grobner bases and the associated graded algebra} for further discussion.)

A general uniform degree bound for left Gr\"obner bases in algebras of solvable type was established by Kredel and Weispfenning \cite{K-W} (using parametric Gr\"obner bases). They showed that, given a monomial ordering $\leq$ on $\mathbb N^N$, there exists a computable function $(d,m)\mapsto B(d,m)$ with the following property:   for every solvable algebra $R$ over some field, generated by $N$ generators whose commutator relations have degree at most $d$, every left ideal of $R$ generated by $m$ elements of $R$ of degree at most $d$ has a Gr\"obner basis (with respect to $\leq$) whose elements have degree at most $B(d,m)$.

In contrast to this, here we are mainly interested in finding explicit, doubly-exponential degree bounds. We follow a road to establish such bounds paved by Dub\'e \cite{Dube}, who gave a self-contained and constructive combinatorial argument for the existence of a doubly-exponential degree bound for Gr\"obner bases in commutative polynomial rings over a field of arbitrary characteristic. Earlier proofs of results of this type (as in \cite{Moller-Mora-1984}) proceed by first homogenizing and then placing the ideal under consideration into {\it generic coordinates.}\/ The drawback of this method is that it seems difficult to adapt it to situations as general as the ones considered here; for one thing, it only works smoothly in characteristic zero.
See \cite{Grigoriev} for the delicacies involved in using automorphisms of the Weyl algebra. (Developing the ideas of the latter paper further,
a doubly-exponential complexity result for Gr\"obner bases in  Weyl algebras over fields of characteristic zero was established in \cite{Chistov-Grigoriev-2007}; the revised journal version of \cite{Chistov-Grigoriev-2007} is \cite{Chistov-Grigoriev-2008}.)

The main new technical tool in \cite{Dube} are decompositions, called  {\it cone decompositions,}\/ of commutative polynomial rings over a field $K$ into a direct sum of finitely many $K$-linear subspaces of a certain type. These decompositions generalize the {\it Stanley decompositions}\/ of a given finitely generated commutative graded $K$-algebra $R$ studied in \cite{SW}. A Stanley decomposition of $R$ encodes a lot of information about $R$; for example, the Hilbert function of $R$ can be easily read off from it. It has been noted in several other places in the literature that Stanley decompositions are ideally suited to avoid the assumption of general position, and, for example, can also be used to circumvent the use of generic hyperplane sections in the proof of Gotzmann's Regularity Theorem \cite{MS}.

The present paper grew out of an attempt by the authors to better understand Dub\'e's article \cite{Dube}. We modified the notions of
cone decompositions and the argument of \cite{Dube} to work for a subclass of the class of algebras of solvable type over an arbitrary coefficient field $K$, namely the ones whose commutation relations are given by {\it quadric}\/ polynomials. (This restriction was necessary in order to be able to freely homogenize the algebras and ideals under consideration.) We refer to Section~\ref{Preliminaries on Algebras over Fields} below for precise definitions, and only note here that this class of algebras includes commutative polynomial rings, as well as Weyl algebras  and the universal enveloping algebra of a finite-dimensional Lie algebra. Many more examples of quadric algebras of solvable type can be found in \cite[Section~I.5]{Li}. (E.g., Clifford algebras, in particular Grassmann algebras, as well as $q$-Heisenberg algebras and the Manin algebra of $2\times 2$-quantum matrices.)

Let now $K$ be a field, and let $R=K\<x\>$ be a quadric $K$-algebra of solvable type with respect to $x=(x_1,\dots,x_N)$ and a monomial ordering $\leq$ of $\N^N$.  Our main theorem is:

\begin{theorem} \label{Main Theorem}
Every left ideal of $R$ generated by elements of degree at most $d$ has a Gr\"obner basis consisting of elements of degree at most
$$D(N,d):=2\left(\frac{d^2}{2}+d\right)^{2^{N-1}}\hspace{-0.7cm}.$$
\end{theorem}

Theorem~\ref{Main Theorem} is deduced from the homogeneous case:  we first show that if $R$ is homogeneous, then the reduced Gr\"obner basis of every left ideal of $R$ generated by homogeneous elements of degree at most $d$ consists of elements of degree at most $D(N-1,d)$, and then obtain the bound in Theorem~\ref{Main Theorem} by dehomogenizing.
Our theorem also yields uniform bounds for reduced Gr\"obner bases in the inhomogeneous case. (See \cite{Latyshev, Weispfenning} for non-explicit uniform degree bounds for reduced Gr\"obner bases in commutative polynomial rings over fields.)
For example, if the monomial ordering $\leq$ is degree-compatible, then the reduced Gr\"obner basis of every left ideal of $R$ generated by elements of degree at most $d$  consists of elements of degree at most $D(N,d)$. (Corollary~\ref{Reduced GB for deg-compatible}.) In the case where the monomial ordering is not degree-compatible, the issues are somewhat more subtle:

\begin{cor}\label{Corollary 1 to Main Theorem}
The elements of the reduced Gr\"obner basis with respect to $\leq$ of every left ideal of $R$ generated by elements of degree at most $d$ have degree at most $$2\,D(N+1,d)\,(N+1)\,N^{N/2}.$$
\end{cor}

It is routine to deduce from Theorem~\ref{Main Theorem}:

\begin{cor} \label{Corollary 2 to Main Theorem}
Suppose the monomial ordering $\leq$ is degree-compatible.
Let $f_1,\dots,f_n\in R$ be of degree at most $d$, and let $f\in R$. If there are $y_1,\dots,y_n\in R$ such that
$$y_1 f_1 + \cdots + y_n f_n = f,$$
then there are such $y_i$ of degree at most $\deg(f)+D(N,d)$. Moreover, the left module of solutions to the linear homogeneous equation
$$y_1 f_1 + \cdots + y_n f_n=0$$
is generated by solutions all of whose components have degree at most $3D(N,d)$.
\end{cor}

For $R=K[x_1,\dots,x_N]$, this corollary is essentially a classical result due to Hermann \cite{Hermann} (corrected and extended by Seidenberg \cite{Seidenberg}). In the case where $R$ is a Weyl algebra, the first statement in this corollary also partly generalizes a result of Grigoriev \cite{Grigoriev} who showed that if a system of linear equations
\begin{equation}\label{System} \tag{$\ast$}
y_1 a_{1j} + \cdots + y_n a_{nj} = b_j \qquad (j=1,\dots,m)
\end{equation}
with coefficients $a_{ij},b_j\in R$ of degree at most $d$
has a solution $(y_1,\dots,y_n)$ in $R$, then this system admits such a solution with $\deg(y_i)\leq (md)^{2^{O(N)}}$ for $i=1,\dots,n$. The methods of \cite{Grigoriev} are quite different from ours, and follow the lead of Hermann  and Seidenberg.
By arguments as in \cite[Corollary~3.4 and Lemma~4.2]{A} one may obtain uniform degree bounds on solutions to systems of linear equations such as \eqref{System} by reduction to Corollary~\ref{Corollary 2 to Main Theorem} (the case $m=1$); however, this yields bounds of the form $d^{2^{O(mN)}}$ that are worse than those obtained by Grigoriev. (Similarly if one tries to use Nagata's ``idealization'' technique as in \cite{Apel-1}.)
Probably,  Corollary~\ref{Corollary 2 to Main Theorem} could be extended from a single linear equation to systems of linear equations with our techniques, by considering Gr\"obner bases of submodules of finitely generated free modules over $R$, as carried out in \cite{Chistov-Grigoriev-2007} in the case of Weyl algebras. 

\medskip

By virtue of an observation from \cite{BGV}, our main theorem, although ostensibly only about one-sided ideals, also has consequences for their two-sided counterparts:

\begin{cor}\label{Corollary 3 to Main Theorem}
Let $f_1,\dots,f_n\in R$ be of degree at most $d$, and let $f\in R$. The two-sided ideal of $R$ generated by $f_1,\dots,f_n$ has a Gr\"obner basis whose elements have degree at most $D(2N,d)$.  If $\leq$ is degree-compatible, and
there are a finite index set $J$ and $y_{ij},z_{ij}\in R$ \textup{(}$i=1,\dots, n$, $j\in J$\textup{)} such that
$$f = \sum_{j\in J} y_{1j} f_1 z_{1j} + \cdots + \sum_{j\in J} y_{nj} f_n z_{nj}$$
then there are such $J$ and $y_{ij}$, $z_{ij}$ with
$$\deg(y_{ij}),\deg(z_{ij})\leq \deg(f)+D(2N,d)\qquad\text{for $i=1,\dots,n$, $j\in J$.}$$
\end{cor}

Weyl algebras are simple (i.e., their only two-sided ideals are the trivial ones). Hence in this case, the previous corollary is vacuous; however, there do exist many non-commutative non-simple algebras satisfying the hypotheses stated before Theorem~\ref{Main Theorem}, for example, among the universal enveloping algebras of finite-dimensional Lie algebras.

\medskip

As shown in \cite{OS}, Gr\"obner basis theory also extends in a straightforward way to certain $K$-algebras closely related to Weyl algebras, namely the rings $R_n(K)$ of partial differential operators with rational functions in $K(x)=K(x_1,\dots,x_n)$ as coefficients. Here $R_n(K)$ is the $K$-algebra generated by $K(x)$ and pairwise distinct symbols $\partial_1,\dots,\partial_n$ subject to the commutation relations
$$\partial_i \partial_j = \partial_j \partial_i, \quad \partial_i c(x)=c(x)\partial_i+\frac{\partial c(x)}{\partial x_i}\qquad (1\leq i\leq j\leq n,\ c(x)\in K(x)).$$
By \cite[Proposition~1.4.13]{SST},  our main theorem implies the existence of a doubly-expo\-nential degree bound for Gr\"obner bases for left ideals in $R_n(K)$:  every left ideal of $R_n(K)$ generated by elements of degree at most $d$ has a Gr\"obner basis with respect to a given monomial ordering $\leq$ of $\N^n$ consisting of elements of degree at most $D(2n,d)$. As above, this result can then be used  to prove an analogue of Corollary~\ref{Corollary 2 to Main Theorem} for $R_n(K)$ (also partially generalizing \cite{Grigoriev}); we omit the details.

\medskip

Assume now that $K$ has characteristic zero, and let $R = A_n(K)$ be the $n$-th Weyl algebra. A proper left ideal $I$ of $R$ is called
{\em holonomic} if the Gelfand-Kirillov dimension of $R/I$ equals $n$, exactly half of the dimension
of $R$. The {\em Bernstein inequality}, versions of which are also known as
the {\em Fundamental Theorems of Algebraic Analysis} (see Theorems 1.4.5 and 1.4.6 of \cite{SST}),
states that $n\leq \dim R/I < 2n$. Therefore, holonomic ideals are proper ideals of the minimal possible dimension,  which brings up an analogy with zero-dimensional ideals in the commutative polynomial setting. Now, there is a bound on the degrees of the elements of a reduced Gr\"obner basis of a zero-dimensional
ideal in a commutative polynomial ring over a field generated in degree at most $d$ that is ({\em single}) exponential. Namely, this is the B\'ezout
bound: $d^n$, where $n$ is the number of indeterminates. (See, e.g., \cite{Lazard}.) Holonomic ideals of $R$ are closely related to zero-dimensional left ideals of the algebra $R_n(K)=K(x)\otimes_{K[x]}R$
of differential operators with coefficients in rational functions: if $I$ is a holonomic ideal of $R$, then the left ideal of $R_n(K)$ generated by $I$ is zero-dimensional, and if conversely $J$ is a zero-dimensional left ideal of $R_n(K)$ then $J\cap R$ is a holonomic ideal; see \cite[Corollary~1.4.14 and Theorem~1.4.15]{SST}. Only a doubly-exponential B\'ezout
bound is known \cite{Grigoriev:weak-Bezout} for zero-dimensional ideals of $R_n(K)$.

So far, to our knowledge, a (single) exponential bound for the degrees of elements in Gr\"obner bases has been produced only for one very special class of holonomic
ideals used in a particular application. These are the {\em GKZ-hypergeometric ideals,}\/ with a homogeneity assumption (cf.~\cite[Corollary 4.1.2]{SST}).
It would be interesting to see if holonomicity (zero-dimensionality) implies a general exponential bound in the algebras $A_n(K)$ ($R_n(K)$, respectively),
as well as whether there is a better bound for ideals of minimal possible dimension in solvable algebras in general.

\medskip

Finally, we would like to mention that although our study is limited to the most frequently used type of bases, Gr\"obner bases, there are other kinds of ``standard bases'' for ideals that may be introduced for algebras of solvable type. For example, \cite{Hausdorf-Seiler-Steinwandt-2002} explores involutive bases in the Weyl algebra.

\subsection{Organization of the paper}
Sections~1 and 2 mainly have preliminary character, and deal with generalities on monomials and $K$-algebras, respectively. In Section~3 we review the fundamentals of Gr\"obner basis theory for algebras of solvable type. In Section~4 we adapt Dub\'e's method to the non-commutative situation, and in Section~5 we prove the main theorem and its corollaries~\ref{Corollary 1 to Main Theorem} and \ref{Corollary 2 to Main Theorem}. In Section~6 we study the two-sided situation.

\subsection{Acknowledgments} We would like to express our gratitude to Dima Grigoriev, Viktor Levandovskyy and the anonymous referees for their numerous suggestions and corrections which helped us to improve the paper.

\section{Monomials and Monomial Ideals} \label{Monomials and Monomial Ideals}

In this section we collect a few notations and conventions concerning multi-indices, monomials and monomial ideals.

\subsection{Multi-indices}
Throughout this note, we let $d$, $m$, $N$ and $n$ range over the set $\N=\{0,1,2,\dots\}$
of natural numbers, and $\alpha$, $\beta$, $\gamma$ and $\lambda$
range over $\N^N$. We let $\N^0=\{0\}$ by convention, and identify
$\N^N$ with the subset $\N^N\times\{0\}$ of $\N^{N+1}$ in the natural way.
We think of the elements of $\N^N$ as {\it multi-indices.}\/
Recall that a {\bf monomial ordering} of $\N^N$ is a total ordering of $\N^N$
compatible with addition in $\N^N$ whose smallest element is $0$.
It is well-known that any monomial
ordering is a well-ordering.
Given total orderings $\leq_1$ of $\N^{N_1}$ and $\leq_2$ of $\N^{N_2}$
($N_1,N_2\in\N$),
the {\bf lexicographic product} of $\leq_1$ and $\leq_2$ is the total
ordering $\leq$ of $\N^{N_1+N_2}=\N^{N_1}\times\N^{N_2}$ defined by
$$(\alpha_1,\beta_1)\leq (\alpha_2,\beta_2) \qquad :\Longleftrightarrow \qquad
\text{$\alpha_1<\alpha_2$, or $\alpha_1=\alpha_2$ and $\beta_1\leq\beta_2$,}$$
for $\alpha_1,\alpha_2\in\N^{N_1}$ and $\beta_1,\beta_2\in\N^{N_2}$.
The lexicographic product of $\leq_1$ and $\leq_2$ extends
$\leq_1$. If $\leq_1$, $\leq_2$ are monomial orderings, then
so is their lexicographic product.
The {\bf lexicographic ordering} of $\N^N$ (the $N$-fold lexicographic product of the usual ordering of $\N$) is denoted by $\leq_\lex$.  For $\alpha=(\alpha_1,\dots,\alpha_N)$ put $\abs{\alpha}:=\alpha_1+\cdots+\alpha_N$.
An ordering $\leq$ of $\N^N$ is said to be {\bf degree-compatible} if
$\abs{\alpha}<\abs{\beta}\Rightarrow \alpha\leq\beta$ for all $\alpha$, $\beta$. An example of a degree-compatible monomial ordering of $\N^N$ is the
{\bf degree-lexicographic ordering}:
$$\alpha \leq_{\dlex} \beta \qquad:\Longleftrightarrow\qquad \text{$\abs{\alpha}<\abs{\beta}$, or $\abs{\alpha}=\abs{\beta}$ and $\alpha \leq_\lex \beta$.}$$
In the rest of this subsection we fix a monomial ordering $\leq$ of $\N^N$.

\medskip

Given a multi-index $\omega$ we define a weight function $\wt=\wt_{\omega}$ (taking non-negative integer values) on the set $\N^N$  by
$$\wt(\alpha):=\omega\cdot\alpha\qquad\text{(inner product of vectors in $\R^N$).}$$
Then for all $\alpha$, $\beta$ we have $\wt(\alpha+\beta) = \wt(\alpha) + \wt(\beta)$,
and if $\omega_i>0$ for each $i$ then
\begin{equation}\label{weight inequality}
\abs{\alpha}\leq \wt(\alpha)\leq ||\omega||\, \abs{\alpha}.
\end{equation}
Here and below, $||\omega||$ denotes the maximum among the absolute values of the components of $\omega$. For a proof of the following quantitative version of a well-known fact about approximating monomial orderings by weight functions see \cite{maschenb-uniform}:

\begin{prop}\label{Approximate leq by wt function}
Let $d$ be given. Then there exists $\omega\in\N^N$ with $||\omega|| \leq 2d(N+1)N^{N/2}$ such that
$$\alpha\leq\beta \quad\Longleftrightarrow\quad \wt_\omega(\alpha) \leq \wt_\omega(\beta)
\qquad\text{for all $\alpha$, $\beta$ with $\abs{\alpha},\abs{\beta}\leq d$.}$$
\end{prop}

\subsection{Monomials and $K$-linear spaces}
In the rest of this section we fix a positive $N$, we let $K$ denote a field, and we let $R$ be a $K$-linear space.
A {\bf monomial basis} of  $R$ is family $\{x^\alpha\}_{\alpha}$ of elements of $R$,  indexed by the multi-indices in $\N^N$, which forms a basis of $R$. Of course, every $K$-linear space of countably infinite dimension has a monomial basis, for every positive $N$, but in the applications in the next sections, a specific monomial basis will always be given to us beforehand.
Thus, in the following we assume that a monomial basis $\{x^\alpha\}_{\alpha}$ of $R$ is fixed.
We call a basis element  $x^\alpha$ of $R$ a {\bf monomial} (of $R$),  and we denote by $x^\diamond$ the set of monomials of $R$.
Every $f\in R$ can be uniquely written in the form
$$f=\sum_\alpha f_\alpha x^\alpha\qquad\text{ where $f_\alpha\in K$, with $f_\alpha=0$
for all but finitely many $\alpha$,}$$
and we define the {\bf support}  of such an $f$ as the set $\supp f$ of all monomials $x^\alpha$ with $f_\alpha\neq 0$.
We have
$x^\alpha\neq x^\beta$ whenever $\alpha\neq\beta$, so we can turn $x^\diamond$ into an ordered monoid by setting $x^\alpha \ast x^\beta = x^{\alpha+\beta}$ and $x^\alpha\leq x^\beta
\Longleftrightarrow\alpha\leq\beta$. The map $\alpha\mapsto x^\alpha\colon\N^N\to x^\diamond$ is then an
isomorphism of ordered monoids.
A tuple of generators of $x^\diamond$ is given by $x=(x_1,\dots,x_N)$ where $x_i=x^{\varepsilon_i}$, with $\varepsilon_i=\text{the $i$-th unit vector in $\N^N$}$.

There is a unique binary operation on $R$ extending the operation $\ast$ on $x^\diamond$ and making the $K$-linear space $R$ into a $K$-algebra. With this multiplication operation, of course, $R$ is nothing but the ring $K[x]$ of polynomials in indeterminates $x=(x_1,\dots,x_N)$ with coefficients from $K$: the unique $K$-linear bijection $K[x]\to R$ which for each multi-index $\alpha$ sends the monomial $x_1^{\alpha_1}\cdots x_N^{\alpha_N}$ of $K[x]$ to the basis element $x^\alpha$ of $R$, is an isomorphism of $K$-algebras. However, in our applications below, the $K$-linear space $R$ will already come equipped with a binary operation making it into a $K$-algebra, and this operation will usually not agree with $\ast$ on $x^\diamond$ (in fact, not even restrict to an operation on $x^\diamond$). In order to clearly separate the combinatorial objects arising in the study of the (generally, non-commutative) $K$-algebras later on, we chose to introduce the extra bit of terminology concerning monomial bases.

A monomial $x^\alpha$  divides a monomial $x^\beta$ (or $x^\beta$ is divisible by $x^\alpha$) if
$x^\beta=x^{\alpha}\ast x^{\gamma}$ for some multi-index $\gamma$; in symbols:
$x^\alpha|x^\beta$.
If $I$ is an ideal of $x^\diamond$, that is, if $x^\alpha\in I \Rightarrow x^\alpha\ast x^\beta\in I$ for all $\alpha$, $\beta$, then there
exist $x^{\alpha(1)},\dots,x^{\alpha(k)}\in I$ such that
each monomial in $I$ is divisible by some $x^{\alpha(i)}$.
(By Dickson's Lemma, \cite[Lemma~1.1]{KR-W}.)
Given monomials $x^\alpha$ and $x^\beta$, the
{\bf least common multiple}  of $x^\alpha$ and $x^\beta$
is the monomial $\lcm(x^\alpha,x^\beta)=x^\gamma$
where $\gamma_i=\max\{\alpha_i,\beta_i\}$ for $i=1,\dots,N$.

Let now $\leq$ be a total ordering of $\N^N$.
Given a non-zero $f\in R$,  there is a unique $\lambda$ with
$$f=f_{\lambda}x^\lambda + \sum_{\alpha<\lambda} f_\alpha x^\alpha, \qquad
f_\lambda\neq 0.$$
We call
$\lc(f)=f_\lambda$ and $\lm(f)=x^\lambda$
the {\bf leading coefficient} respectively {\bf leading monomial} of $f$ with respect to $\leq$.
It is convenient to define $\lm(0):=0$ and extend $\leq$ to a total
ordering on the set $x^\diamond\cup\{0\}$ by declaring
$0<x^\alpha$ for all $\alpha$. We also declare $\lc(0):=0$.
We extend the notation $\lm$ to subsets of $R$ by a slight abuse: for $S\subseteq R$ put
$$\lm(S) := \big\{ \lm(f) : 0\neq f\in S\big\}\subseteq x^\diamond.$$

\subsection{Monomial cones and monomial ideals}
By abuse of notation, we write $y\subseteq x$ to indicate that $y$ is a subset of $\{x_1,\dots,x_N\}$, and for $y\subseteq x$ we let $y^\diamond$ be the submonoid of $(x^\diamond,{\ast})$ generated by $y$. (So $\emptyset^\diamond=\{1\}$.)

A {\bf monomial cone} defined by a pair $(w,y)$, where $w\in x^\diamond$ and $y\subseteq x$, is the $K$-linear subspace $C(w,y)$ of $R$ generated by $w\ast y^\diamond$.
Note that $C(w,\emptyset)=Kw$ for every $w\in x^\diamond$, and $C(1,x)=R$.
Also, if $y\subseteq y'\subseteq x$ then $C(w,y)\subseteq C(w,y')$. We refer to \cite[Section~3]{Dube} for how to represent monomial cones graphically in the (slightly misleading) case $N=2$.
If we identify $R$ with the commutative polynomial ring $R=K[x]$ as explained above, then $C(w,y)$ is nothing but the $K$-linear subspace $wK[y]$ of $K[x]$.

We say that a $K$-linear subspace $I$ of $R$ is a {\bf monomial ideal} if $I$ is
spanned by monomials,
and $C(w,x)\subseteq I$ for all monomials $w\in I$. (Hence, if $R=K[x]$, then $I$ is a monomial ideal of $K[x]$ in the usual sense of the word.)
A set of {\bf generators}
for a monomial ideal $I$ of $R$ is defined to be a set of monomials $F$ such that $I = \sum_{w\in F} C(w,x)$ (so the set $F\ast x^\diamond$ generates $I$ as a $K$-linear space). A $K$-linear subspace of $R$ is a monomial ideal if and only if the set of monomials in $I$ is an ideal of $(x^\diamond,{\ast})$.
Every $K$-subspace of $R$ generated by monomials has a unique minimal set of generators, which is finite.

Given a monomial ideal $I$ of $R$ and a monomial $w$ we put
$$(I:w) := \text{the $K$-linear subspace of $R$ generated by $\{ v\in x^\diamond : w\ast v\in I\}$,}$$
a monomial ideal of $R$ containing $I$.

Let now $M$ be a $K$-linear subspace of $R$ generated by monomials, and let $I$ be a monomial ideal of $R$. Then the $K$-linear subspace $M\cap I$ of $M$ has a natural complement:
$$M=(M\cap I)\oplus \nf_I(M),$$
where $\nf_I(M)$ denotes the $K$-linear subspace of $R$ generated by the monomials in $M\setminus I$.

\section{Preliminaries on Algebras over Fields} \label{Preliminaries on Algebras over Fields}

In this section we let $K$ be a field (of arbitrary characteristic). All $K$-algebras are  assumed to be associative with unit $1$. Given a subset $G$ of
a $K$-algebra $R$ we denote by $(G)$ the {\it left}\/ ideal of $R$ generated by $G$.
We also let $\leq$ be a monomial ordering of $\N^N$.

\subsection{Multi-filtered  $K$-algebras and modules}
A {\bf multi-filtra\-tion} on $R$ (indexed by $\N^N$)
is an increasing (with respect to $\leq$) family  of $K$-linear subspaces $\big\{R_{(\leq\alpha)}\big\}_{\alpha}$ of $R$ whose union is $R$ and such that $1\in R_{(\leq 0)}$ and $R_{(\leq\alpha)}\cdot R_{(\leq\beta)}\subseteq R_{(\leq \alpha+\beta)}$.
A {\bf multi-filtered $K$-algebra} is a $K$-algebra equipped with a
multi-filtration.
Suppose $R$ is a multi-filtered $K$-algebra. A {\bf multi-filtra\-tion} on a left $R$-module $M$ (indexed by $\N^N$)
is an increasing family of $K$-linear subspaces  $\big\{M_{(\leq\alpha)}\big\}_{\alpha}$ of $M$ which exhausts $M$ and such that $R_{(\leq\alpha)}\cdot M_{(\leq\beta)}\subseteq M_{(\leq \alpha+\beta)}$. A {\bf multi-filtered left $R$-module} is a left $R$-module equipped with a
multi-filtration. Suppose that  $M$ is a multi-filtered  left $R$-module.
For every $\alpha$ the set $M_{(<\alpha)}:=\bigcup_{\beta<\alpha} M_{(\leq\alpha)}$ is a $K$-linear
subspace of $M$. Here $M_{(<0)}:=\{0\}$ by convention.
For every non-zero
$f\in M$ there exists a unique $\alpha$ with $f\in M_{(\leq\alpha)}\setminus
M_{(<\alpha)}$, and we call $\alpha=\deg(f)$ the {\bf degree} of $f$.
Given a left $R$-submodule $M'$ of $M$, we always construe $M'$ as a multi-filtered left $R$-module by means of the multi-filtration $\{M'_{(\leq\alpha)}\}_\alpha$ given by $M'_{(\leq\alpha)}:=M'\cap M_{(\leq\alpha)}$ for every $\alpha$, and we make the quotient $M/M'$ into a multi-filtered left $R$-module by the multi-filtration induced on $M/M'$ from $M$ by the natural surjection $M\to M/M'$, given by
$(M/M')_{(\leq \alpha)}:=(M_{(\leq \alpha)}+M')/M'$ for every $\alpha$.
For a two-sided ideal $I$ of $R$, the induced filtration makes $R/I$ a multi-filtered $K$-algebra.

\subsection{Multi-graded $K$-algebras and modules}
A {\bf multi-grading} on  $R$ (indexed by $\N^N$) is
a family $\big\{R_{(\alpha)}\big\}_{\alpha}$ of $K$-linear subspaces of $R$
such that $R=\bigoplus_{\alpha} R_{(\alpha)}$ (internal direct sum of
$K$-linear subspaces of $R$) and $R_{(\alpha)}\cdot R_{(\beta)}\subseteq R_{(\alpha+\beta)}$
for all multi-indices $\alpha$, $\beta$.
A $K$-algebra equipped with a
multi-grading is called a {\bf multi-graded $K$-algebra}. Suppose $R$ is multi-graded.
A {\bf multi-grading} on a left $R$-module $M$ (indexed by $\N^N$) is
a family $\big\{M_{(\alpha)}\big\}_{\alpha}$ of $K$-linear subspaces of $M$
such that  $M=\bigoplus_{\alpha} M_{(\alpha)}$ and $R_{(\alpha)}\cdot M_{(\beta)}\subseteq M_{(\alpha+\beta)}$
for all $\alpha$, $\beta$.
A left $R$-module equipped with a
multi-grading is called a {\bf multi-graded left $R$-module}. Let $M$ be a multi-graded left $R$-module.
We call the $K$-linear subspace
$M_{(\alpha)}$ of $M$ the {\bf homogeneous component}
of degree $\alpha$ of $M$.
We always view $R$ as a multi-filtered $K$-algebra, and $M$ as a multi-filtered left $R$-module by means of the natural multi-filtrations $\big\{R_{(\leq\alpha)}\big\}_{\alpha}$ and
$\big\{M_{(\leq\alpha)}\big\}_{\alpha}$ given by $$R_{(\leq\alpha)} := \bigoplus_{\beta\leq\alpha}
R_{(\beta)},\quad M_{(\leq\alpha)} := \bigoplus_{\beta\leq\alpha}
M_{(\beta)}\qquad\text{ for every $\alpha$.}$$
Every $f\in M$ has a unique
representation in the form $f=\sum_\alpha f_{(\alpha)}$
where $f_{(\alpha)}\in M_{(\alpha)}$ for all $\alpha$, and $f_{(\alpha)}=0$
for all but finitely many $\alpha$.
We call $f_{(\alpha)}$
the homogeneous component of degree $\alpha$ of  $f$.
Similarly, given
a $K$-linear subspace $V$ of $M$ which is homogeneous (i.e., for
$f\in M$ we have
$f\in V$ if and only if $f_{(\alpha)}\in V$ for each $\alpha$),
the homogeneous component of degree $\alpha$ of $V$
is denoted by
$V_{(\alpha)}:=V\cap M_{(\alpha)}$, so
$$V=\bigoplus_\alpha V_{(\alpha)} \qquad\text{(internal direct sum of $K$-linear subspaces of $M$).}$$
If $M'$ is a homogeneous left $R$-submodule of $M$, then the $M'_{(\alpha)}$ furnish $M'$ with a multi-grading, and
we make $M/M'$ into a multi-graded left $R$-module by the multi-grading
induced from $M$, given by
$(M/M')_{(\alpha)}:=(M_{(\alpha)}+M')/M'$ for every $\alpha$.
The multi-filtration of $M/M'$ associated to this multi-grading
agrees with the multi-fil\-tra\-tion of $M/M'$ induced from the multi-filtered left $R$-module $M$.
If $I$ is a two-sided ideal of $R$, then $R/I$ a multi-graded $K$-algebra by means of the induced multi-grading.

\subsection{The associated multi-graded algebra}
Suppose $R$ is multi-filtered, and let $M$ be
a multi-filtered left $R$-module $M$. The
left $R$-module $$\gr M=
\bigoplus_\alpha\ (\gr M)_{(\alpha)} \qquad\text{with $(\gr M)_{(\alpha)}=M_{(\leq \alpha)}/M_{(<\alpha)}$}$$
is a multi-graded left $\gr R$-module in a natural way, called the {\bf multi-graded left $\gr R$-module associated to $M$.}
(For $M=R$ we obtain a multi-graded $K$-algebra called
the {\bf multi-graded $K$-algebra  $\gr R$ associated to $R$.})
For non-zero $f\in M$ of degree $\alpha$,
$$\gr f:=f+M_{(<\alpha)}\in (\gr M)_{(\alpha)}$$ is the {\bf initial form} (or {\bf symbol}) of $f$, and  $\gr 0:=0\in\gr M$. Given a left $R$-submodule $M'$ of $M$, the inclusion $M'\to M$ induces an embedding $\gr M'\to\gr M$ of multi-graded left $R$-modules, and we identify $\gr M'$ with its image under this embedding.

\subsection{The Rees algebra} Suppose $R$ is multi-filtered.
The {\bf Rees algebra}
of $R$ is the multi-graded $K$-algebra $$R^*=\bigoplus_\alpha\ (R^*)_{(\alpha)} \qquad
\text{with $(R^*)_{(\alpha)}=R_{(\leq\alpha)}.$}$$
For a non-zero element $f$ of $R$ of degree $\alpha$ we let $f^* :=
f\in (R^*)_{(\alpha)}$ be the {\bf homogenization} of
$f$; by convention $0^*:=0$.
Let $I$ be a two-sided ideal of $R$. We let
$I^*$ be the two-sided ideal of $R^*$ generated by all $f^*$ with $f\in I$;
the ideal $I^*$ is homogeneous, and is called the {\bf homogenization} of
$I$.
The natural surjection $R\to R/I$ is a morphism of
multi-filtered $K$-algebras which induces a surjective
morphism $R^*\to (R/I)^*$ of multi-graded $K$-algebras
whose kernel is $I^*$; the induced
morphism $R^*/I^*\to (R/I)^*$ is an isomorphism of multi-graded $K$-algebras.
The natural inclusions $(R^*)_{(\alpha)}=R_{(\leq\alpha)}\subseteq R$ combine to a
$K$-linear map $h\mapsto h_*\colon R^*\to R$
which is a surjective morphism of multi-graded $K$-algebras
satisfying $(f^*)_*=f$ for all $f\in R$.  For $h\in R^*$ the element $h_*$ of $R$ is
called the {\bf dehomogenization} of $h$. We extend this notation to
subsets of $R^*$: $H_*:=\{h_*: h\in H\}$ for $H\subseteq R^*$.
If $J$ is a left ideal of $R^*$, then $J_*$ is a left ideal of $R$.
Hence if $H\subseteq R^*$ then $(H)_*=(H_*)$.

\subsection{Filtered and graded algebras}
By a {\bf filtered $K$-algebra} we will mean an
multi-filtered algebra with filtration indexed by $\N$, and similarly
a multi-graded $K$-algebra whose grading is indexed by $\N$ is just called a {\bf graded $K$-algebra.} Analogous terminology is used in the case of left $R$-modules. (Most of our multi-filtered or multi-graded objects will actually be filtered, respectively graded; we introduced the more general concepts in order to be able to speak about the ``fine filtration'' (Lemma~\ref{gr R is semi-commutative}) of an  algebra of solvable type.

Suppose $R=\bigcup_d R_{(\leq d)}$ is a filtered $K$-algebra. We denote by $t$ the {\bf canonical element} of $R^*$, that is,
the unit $1$ of $R$, considered as an element of
$(R^*)_{(1)}=R_{(\leq 1)}$.
In this case the natural surjections $$(R^*)_{(d)}=R_{(\leq d)}\to
R_{(\leq d)}/R_{(<d)}=(\gr R)_{(d)}$$ combine to a surjective $K$-algebra morphism $R^*\to\gr R$ which has
kernel $R^*t$ and hence induces an isomorphism of graded $K$-algebras
$R^*/R^*t\overset{\cong}{\longrightarrow}\gr R$.

\subsection{Homogenization of graded algebras}
Suppose now that $R=\bigoplus_d R_{(d)}$ is a graded $K$-algebra.
We make the ring
$R[T]$ of polynomials in one commuting indeterminate $T$ over $R$
into a graded $K$-algebra using the grading
$$R[T]=\bigoplus_d R[T]_d\qquad\text{ with $R[T]_{(d)} := \bigoplus_{i+j=d} R_{(i)}T^j$.}$$
The $K$-linear map $R[T]\to R^*$ with $fT^j\mapsto ft^{j}$ for
all $f\in R_{(i)}$ and $i,j\in\N$ is an isomorphism of graded $K$-algebras.
In the following we always identify the Rees algebra of a graded $K$-algebra $R$ with
the graded $K$-algebra $R[T]$. Then the canonical element of $R^*$ is $T$,
and for non-zero $f\in R$ of degree $d$ we  have $$f^* =
\sum_{i=0}^d f_{(i)}T^{d-i}\in (R^*)_{(d)},$$
and for $h=\sum_{i=0}^n h_i T^i\in R^*$ we get
$h_*=\sum_{i=0}^n h_i \in R$.

\subsection{Non-commutative polynomials}
In the following we let $X=(X_1,\dots,X_N)$ be a tuple of $N$ distinct indeterminates over $K$
and denote by $X^*$ the free monoid generated by $\{X_1,\dots,X_N\}$.
The free $K$-algebra
$K\<X\>=K\<X_1,\dots,X_N\>$ generated by $X$
(that is, the monoid algebra of $X^*$ over $K$) has a natural grading
$$K\<X\>=\bigoplus_d K\<X\>_{(d)}$$ defined by the length of words in
$X^*$.
Let $I$ be a  two-sided ideal of $K\<X\>$.
The  $K$-algebra $R=K\<X\>/I$ is generated by the cosets $X_i+I$ ($i=1,\dots,N$).
Let $T$ be an indeterminate over $K$ distinct from $X_1,\dots,X_N$.
We identify the Rees algebra $K\<X\>^*$ of $K\<X\>$ with the graded $K$-algebra $K\<X\>[T]$
as explained in the previous subsections; similarly, the Rees algebra
$R^*$ of $R$ will be identified with
$K\<X\>^*/I^*=K\<X\>[T]/I^*$.
For a non-zero $f\in K\<X\>$ of degree $d$
we define the homogeneous polynomial
\begin{equation}\label{fhom}
f^\hom := \sum_{i=0}^d f_{(i)}T^{d-i}\in K\<X,T\>.
\end{equation}
The two-sided ideal $I^\hom$ of $K\<X,T\>$
generated by $f^\hom$ for non-zero $f\in I$ and the polynomials $X_iT-TX_i$ ($i=1,\dots,N$) is homogeneous, and
the natural $K$-linear map $K\<X,T\>\to K\<X\>[T]$
induces an isomorphism of graded $K$-algebras
\begin{equation}\label{Iso-homog}
K\<X,T\>/I^\hom \overset{\cong}{\longrightarrow} R^*=K\<X\>[T]/I^*.
\end{equation}

\subsection{Affine algebras}
{\em In the rest of this section, we let $R$ be a finitely generated $K$-al\-ge\-bra and we fix a tuple $x=(x_1,\dots,x_N)$ of elements of $R$.}\/
For a multi-index $\alpha=(\alpha_1,\dots,\alpha_N)$ put $x^\alpha:=x_1^{\alpha_1}\cdots x_N^{\alpha_N}$. We say that the $K$-algebra $R$ is {\bf affine}
with respect to $x$ if the family $\{x^\alpha\}_\alpha$ is a monomial basis of the $K$-linear space $R$. (Note that then $x_1,\dots,x_N$ generate $R$ as a $K$-algebra.)
Usually, we obtain affine $K$-algebras by specifying a
{\bf commutation system} in $K\<X\>$, that is, a family ${\cal R}=(R_{ij})_{1\leq i<j\leq N}$ of
${N\choose 2}$ polynomials
\begin{multline}\label{Rij}
R_{ij} = X_jX_i-c_{ij}X_iX_j-P_{ij} \\
\text{where $0\neq c_{ij}\in K$ and $P_{ij}\in \bigoplus_\alpha KX^\alpha$ for $1\leq i<j\leq N$.}
\end{multline}
Let ${\cal R}=(R_{ij})$ be a commutation system and $I=I(\cal R)$ be
the two-sided ideal of $K\<X\>$ generated by the polynomials
$R_{ij}$ ($1\leq i<j\leq N$), and suppose $R=K\<X\>/I$ with $x_i=X_i+I$ ($i=1,\dots,N$). We say that the finitely presented
$K$-algebra $R$ is {\bf defined by $\cal R$.} We construe $K\<X\>$ as a filtered $K$-algebra
via filtration by degree of polynomials in $K\<X\>$, and we equip
$R$ with the filtration induced by the natural surjection
$K\<X\>\to K\<X\>/I=R$, called the {\bf standard filtration} of $R$ (with respect to $x_1,\dots,x_N$).
If $R$ turns out to be affine, then the generators $x_1,\dots,x_N$ of the $K$-algebra $R$ have degree $1$.



\begin{examples}\label{Affine, Examples}
Affineness of $K$-algebras may be shown using the techniques in \cite{Bergman}, and also with
Mora's theory \cite{Mora} of Gr\"obner bases for two-sided ideals in $K\<X\>$  (cf.~\cite[Theorem~1.11]{KR-W}). Some prominent examples for affine $K$-algebras:

\begin{enumerate}
\item A $K$-algebra is called {\bf semi-commutative} if for every pair $f,g$ of its elements there is a non-zero $c\in K$ with $fg=cgf$. If $P_{ij}=0$ for $1\leq i<j\leq N$  in \eqref{Rij}, then the $K$-algebra defined by $\cal R$ is affine and semi-commutative.
If in addition $c_{ij}=1$ for $1\leq i<j\leq N$,
then the $K$-algebra defined by $\cal R$ is naturally isomorphic to
the $K$-algebra $K[x]=K[x_1,\dots,x_N]$ of commutative polynomials in the tuple of indeterminates $x=(x_1, \dots, x_N)$ with coefficients in $K$.
\item
The $n$-th Weyl algebra $A_n(K)$ over $K$
is the $K$-algebra generated by $N=2n$ generators $x_1,\dots,x_n,\partial_1,
\dots,\partial_n$ subject to the relations
$$
\begin{array}{llll}
x_jx_i = x_ix_j, & \partial_j\partial_i=\partial_i\partial j & &\text{for $1\leq i<j\leq n$,} \\
\partial_j x_i = x_i\partial_j & & &\text{for $1\leq i,j\leq n$, $i\neq j$,} \\
\partial_i x_i = x_i\partial_i + 1 & & &\text{for $1\leq i\leq n$.}
\end{array}$$
The $K$-algebra $A_n(K)$ is
affine with respect to the generating tuple $(x,\partial):=(x_1,\dots,x_n,\partial_1,\dots,\partial_n)$. The standard filtration of $A_n(K)$ is also known as the Bernstein filtration of $A_n(K)$.
\item Let $\frak g$ be a Lie algebra over $K$ of dimension $n$, and let
$\{x_1,\dots,x_N\}$ be a basis of $\frak g$.
The universal
enveloping algebra of $\frak g$ is a $K$-algebra $U(\frak g)$
which contains $\frak g$ as $K$-linear subspace and is
generated by $x_1,\dots,x_N$ subject to the relations
$$x_jx_i = x_ix_j - [x_j,x_i]_{\frak g} \quad\text{for $1\leq i<j\leq N$.}$$
The fact that $U(\frak g)$ is affine with respect to
the tuple $(x_1,\dots,x_N)$  is known as the
Poincar\'e-Birkhoff-Witt Theorem \cite[Theorem~3.1]{Bergman}. (Hence affine algebras are also known as ``algebras with PBW-basis.'')
\end{enumerate}
\end{examples}

We say that a commutation system $\cal R=(R_{ij})$ as above is {\bf quadric} if
every polynomial $P_{ij}$ has degree $\leq 2$,
{\bf linear} if every $P_{ij}$ has degree $\leq 1$, and
{\bf homogeneous} if all $R_{ij}$ are either zero or homogeneous (necessarily of degree $2$). All examples
of affine $K$-algebras given above are defined by linear commutation
systems.

\subsection{Algebras of solvable type}
The definition below is due to
Kandri-Rody and Weis\-pfenning \cite{KR-W}. Recall that $\leq$ denotes a monomial ordering of $\N^N$.

\begin{definition}\label{Solvable Type}
The $K$-algebra $R$ is said to be {\bf of solvable type}
with respect to the fixed monomial ordering $\leq$ of $\N^N$
and the tuple $x=(x_1,\dots,x_N)\in R^N$ if $R$ is affine with respect to $x$, and
for $1\leq i<j\leq N$ there are
$c_{ij}\in K$, $c_{ij}\neq 0$, and $p_{ij}\in R$ such that
$$x_jx_i = c_{ij} x_ix_j + p_{ij} \quad \text{and} \quad
\lm(p_{ij})<x_ix_j.$$
(Note that the $c_{ij}$ and $p_{ij}$ are then uniquely determined.)
\end{definition}

If $R$ is of solvable type with respect to $\leq$ and $x$, then (cf.~\cite[Lemma~1.4]{KR-W})
\begin{equation}\label{lm multiplication}
\lm(f\cdot g)=\lm(f)\ast\lm(g) \qquad\text{for non-zero $f,g\in R$.}
\end{equation}
In particular, $R$ is an integral domain.
If $R$ is semi-commutative, then $R$ is of solvable type with respect to $x$ and every monomial ordering of $\N^N$, and each homogeneous component $R_{(\alpha)}$ of $R$ has the form $R_{(\alpha)}=Kx^\alpha$. Therefore:

\begin{lemma}\label{gr R is semi-commutative}
Suppose $R$ is of solvable type with respect to $\leq$ and $x$. Then
$$R_{(\leq\alpha)} := \bigoplus_{\beta\leq\alpha}
Kx^\beta$$
defines a multi-filtration of $R$, and its associated multi-graded
$K$-algebra $\gr_\leq R$ is semi-com\-mutative with respect to $\leq$ and $\xi=(\xi_1,\dots,\xi_N)$, where $\xi_i:=\gr_\leq x_i$ for $i=1,\dots,N$. If $c_{ij}=1$ for $1\leq i<j\leq N$, then $\gr_{\leq} R=K[\xi]$ is commutative.
\end{lemma}

Here is a way of constructing $K$-algebras of solvable type \cite[Theorem~1.7]{KR-W}:

\begin{prop}\label{Solvable Type, Construction}
Let $\cal R=(R_{ij})$ be a commutation system with $R_{ij}$ as in \eqref{Rij},
let $I=I(\cal R)$, and suppose $R=K\<X\>/I$ with
$x_i=X_i+I$ for $1\leq i\leq N$. Then $R$ is of solvable
type with respect to the monomial  ordering $\leq$ and
the tuple $x=(x_1,\dots,x_N)$ of generators for $R$
if and only if the following two conditions are satisfied:
\begin{enumerate}
\item $\lm(P_{ij})<\lm(X_iX_j)$ for $1\leq i<j\leq N$, and
\item $I\cap \bigoplus_\alpha KX^\alpha=\{0\}$.
\end{enumerate}
\end{prop}

\begin{remark} \label{Solvable Type, Construction, Remark}
Suppose that $R$ is affine with respect
to $\leq$ and $x$, and let $\pi\colon K\<X\>\to R$ be the surjective $K$-algebra
morphism with $X_i\mapsto x_i$ for $i=1,\dots,N$. Let
$\cal R=(R_{ij})$ be a commutation system as in \eqref{Rij}
satisfying condition (1) in Proposition~\ref{Solvable Type, Construction} and
with $\ker\pi$ containing $I=I(\cal R)$. Then
$I=\ker\pi$, so $R$ is of solvable type with respect to $\leq$ and
$x$.
(Note that $\ker\pi\cap \bigoplus_\alpha KX^\alpha=\{0\}$
since $R$ is affine; in particular, $I\cap \bigoplus_\alpha KX^\alpha=\{0\}$,
hence $K\<X\>=I\oplus \bigoplus_\alpha KX^\alpha$
by Proposition~\ref{Solvable Type, Construction}, and thus $I=\ker\pi$.)
\end{remark}

Every $K$-algebra of solvable type arises as described in
Proposition~\ref{Solvable Type, Construction}:
Suppose $R=K\<x\>$ is
of solvable type as in Definition~\ref{Solvable Type}; let $\pi$ be as
in Remark~\ref{Solvable Type, Construction, Remark},
for $1\leq i<j\leq N$ let $P_{ij}$ be the unique polynomial in $\bigoplus_\alpha KX^\alpha$ with $\pi(P_{ij})=p_{ij}$, and define the commutation system
${\cal R}=(R_{ij})$ as in \eqref{Rij}.
Then clearly $\ker\pi$ contains
$I=I(\cal R)$. So $\ker\pi=I$ by the preceding remark, and $\pi$ induces an isomorphism $K\<X\>/I\to R$.
Hence we may define properties of a $K$-algebra of solvable type in terms
of the unique commutation system defining it. For example, we say that
a $K$-algebra of solvable type is {\bf quadric} or {\bf homogeneous} if its defining commutation system is quadric or homogeneous, respectively. If $R$ is of solvable type with respect to a degree-compatible monomial ordering, then $R$ is quadric.

Condition (1) in the previous proposition
automatically holds if $P_{ij}\in K$ for $1\leq i<j\leq N$, or
if $\leq$ is degree-compatible and $\deg P_{ij}<2$ for $1\leq i<j\leq N$.
Hence  the
$n$-th Weyl algebra $A_n(K)$ over $K$ is of solvable type with respect to the generating tuple $(x,\partial)$ and
{\it every}\/ monomial ordering of $\N^{2n}$. Similarly,
the universal enveloping algebra of an $N$-dimensional Lie algebra
over $K$ is of solvable type with respect to the generating tuple $x$ and every monomial ordering of $\N^N$. The only
commutative $K$-algebra of solvable type with respect to $x$ is the
commutative polynomial ring $K[x_1,\dots,x_N]$, which is of solvable type with respect to every monomial ordering of $\N^N$. All of those examples are quadric.

\begin{lemma}\label{mod t}
Suppose that $N>0$ and $x_N$ is in the center of $R$. Let $S=R/Rx_N$, and for $i=1,\dots,N-1$ let
$y_i$ be the image of $x_i$ under the natural surjection $R\to S$.
\begin{enumerate}
\item If  $R$ is
affine with respect to $x$, then $S$ is affine with respect to
$y=(y_1,\dots,y_{N-1})$.
\item If $R$ is of solvable type with respect to $\leq$ and the tuple $x$, then $S$ is of solvable type with respect to
the restriction of $\leq$ to $\N^{N-1}$ and $y$, and if in addition $R$ is quadric \textup{(}homogeneous\textup{)}, then $S$ is quadric \textup{(}homogeneous, respectively\textup{)}.
\end{enumerate}
\end{lemma}
\begin{proof}
Part (1) is clear. For (2), suppose $R$ is of
solvable type with respect to   $\leq$ and $x$.
Let $\cal R=(R_{ij})_{1\leq i<j\leq N}$ be the commutation system
in $K\<X\>$ defining
$R$. Let $Y=(Y_1,\dots,Y_{N-1})$ be a tuple of distinct indeterminates over $K$.
The commutation system $\cal S=(S_{ij})_{1\leq i<j<N}$
in $K\<Y\>$
with $S_{ij}:=R_{ij}(Y,0)$ for $1\leq i<j< N$
satisfies condition (1) in Proposition~\ref{Solvable Type, Construction},
and $I(\cal S)$ is contained in the kernel of the $K$-algebra
morphism $K\<Y\>\to S$ with $Y_i\mapsto y_i$ for $i=1,\dots,N-1$. Hence
by (1) and Remark~\ref{Solvable Type, Construction, Remark}, $S$  is of
solvable type with respect to the restriction of $\leq$ to $\N^{N-1}$ and $y$.
If $\cal R$ is quadric (homogeneous) then $\cal S$ clearly is quadric
(homogeneous, respectively).
\end{proof}

\subsection{Quadric algebras of solvable type}
\label{Quadric algebras of solvable type}
{\it In the rest of this section,  $\pi\colon K\<X\>\to R$ is
the $K$-algebra morphism with $\pi(X_i)=x_i$.
Also let $\cal R=(R_{ij})$ be a commutation system defining $R=K\<x\>$, with $R_{ij}$ as in \eqref{Rij}, and we assume that $R$ is quadric of solvable type with respect to
$\leq$ and $x$. We put $p_{ij}:=\pi(P_{ij})$.}\/
We have
$\lm(\pi(v))=\lm(\pi(w))$ for all words $v,w\in \<X\>$ which are rearrangements of
each other, by \eqref{lm multiplication}. This observation is crucial
for the proof of the next lemma, to be used in the following subsection:

\begin{lemma}\label{R_leq d}
For every $d$ we have
$$R_{(\leq d)} = \bigoplus_{\abs{\alpha}\leq d} Kx^\alpha.$$
\end{lemma}
\begin{proof}
For a word $w=X_{i_1}\cdots X_{i_m}\in X^*$ with $i_1,\dots,i_m\in\{1,\dots,N\}$ we define the ``misordering index'' $i(w)$ of $w$ as the number of pairs $(k,l)$ with $1\leq k<l\leq m$ and $i_k>i_l$. We equip $\N^{N+1}=\N^N\times\N$ with the lexicographic product of
the given monomial ordering $\leq$ of $\N^N$ and the usual ordering of $\N$.
It suffices to show, by induction on $(\alpha,i)\in \N^N\times\N$, that
every $w\in \<X\>$ with $\lm(\pi(w))=x^\alpha$ and the misordering index $i(w)=i$ belongs to $I(\cal R)+\bigoplus_{\abs{\beta}\leq d} KX^\beta$ where $d=\text{length of $w$}$.
If $i(w)=0$ then $w\in
\bigoplus_{\abs{\beta}\leq d} KX^\beta$, and there is nothing to show; so
suppose $i(w)>0$ (in particular, $d>0$). Then
there are $i$, $j$ and $u$, $v$  with
$i<j$, $w=uX_jX_iv$ and $i(u)=0$. We have $uR_{ij}v\in I(\cal R)$ and
$$w=c_{ij}uX_iX_jv+uP_{ij}v+uR_{ij}v.$$
We also have
$\lm(\pi(uX_iX_jv))=\lm(\pi(w))$ and $i(uX_iX_jv)=i(w)-1$, and moreover
$\lm(\pi(uP_{ij}v))<\lm(\pi(w))$ and $\deg(uP_{ij}v)\leq d$ since $\cal R$ is quadric. Thus
by inductive hypothesis,
$uX_iX_jv$ and $uP_{ij}v$ are elements of $I(\cal R)+\bigoplus_{\abs{\beta}\leq d} KX^\beta$; hence so is $w$.
\end{proof}

\subsection{Homogenization and homogeneous algebras of solvable type}
\label{Homogenization}

Let $T$ be an indeterminate over $K$ distinct from
$X_1,\dots,X_N$.
In the following we identify the Rees algebra $R^*$ of $R$ with
the graded $K$-algebra $K\<X,T\>/I(\cal R)^\hom$ via the isomorphism
\eqref{Iso-homog}.
Then the canonical element of $R^*$ is
$t=T+I(\cal R)^\hom$, and the $K$-algebra $R^*$ is generated
by $x_1^*,\dots,x_N^*,t\in (R^*)_{(1)}$, where
$x_i^*=X_i+I(\cal R)^\hom$ is  the homogenization of $x_i$ ($i=1,\dots,N$).
Let $x^*:=(x_1^*,\dots,x_N^*)$. By Lemma~\ref{R_leq d}, for every $d$ we have
$$(R^*)_{(d)} = \bigoplus_{\abs{\alpha}\leq d} K\, (x^*)^\alpha t^{d-\abs{\alpha}}.$$
In particular, the $K$-algebra $R^*$ is affine with respect to $(x^*,t)$. In fact:

\begin{cor}\label{Rees}
The Rees algebra $R^*$ of
$R$ is  homogeneous of solvable type with respect to the lexicographic product
$\leq^*$ of  the monomial ordering $\leq$ of $\N^N$ and the usual ordering
of $\N$, and the generating tuple $(x^*,t)$.
\end{cor}
\begin{proof}
We construct a homogeneous
commutation system $\cal R^\hom$ in $K\<X,T\>$ by enlarging
the family $(R_{ij}^\hom)_{1\leq i<j\leq N}$  by the
polynomials $X_iT-TX_i$ ($i=1,\dots,N$). (See \eqref{fhom} for the definition of $R_{ij}^\hom$.)
Then $\cal R^\hom$ satisfies
condition (1) in Proposition~\ref{Solvable Type, Construction} (by choice of $\leq^*$).
Clearly the surjective $K$-algebra morphism $K\<X,T\>\to R^*$ with $X_i\mapsto x_i^*$  and $T\mapsto t$ sends every polynomial in $I(\cal R^\hom)$ to zero, hence induces
an isomorphism $K\<X,T\>/I(\cal R^\hom)\to R^*$ by Remark~\ref{Solvable Type, Construction, Remark}. Thus
$R^*$ is of solvable type as claimed.
\end{proof}

In the following, by abuse of notation, we denote the homogenization
$x_i^*\in R^*$ of $x_i\in R$ also just by $x_i$, for $i=1,\dots,N$.
So the homogenization of $f\in R$ of degree $d$ is
$$f^* = \sum_\alpha f_\alpha  x^\alpha t^{d-\abs{\alpha}} \in (R^*)_{(d)},$$
and for every  $\alpha$ and $i\in\N$ the dehomogenization of
$x^\alpha t^i$ is given by $(x^\alpha t^i)_*=x^\alpha$.

\begin{examples}\label{Rees-Examples}
\mbox{}
\begin{enumerate}
\item The Rees algebra of the commutative polynomial ring $K[x_1,\dots,x_N]$ is
the polynomial ring $K[x_1,\dots,x_N,t]$ equipped with its
usual grading by (total) degree.
\item If $R=A_n(K)$, then $R^*$ is the graded $K$-algebra generated by
$2n+1$ generators
$x_1,\dots,x_n,\partial_1,
\dots,\partial_n,t$ subject to the homogeneous relations
$$
\begin{array}{llll}
x_jx_i = x_ix_j, & \partial_j\partial_i=\partial_i\partial j & &\text{for $1\leq i<j\leq n$,} \\
\partial_j x_i = x_i\partial_j & & &\text{for $1\leq i,j\leq n$, $i\neq j$,} \\
\partial_i x_i = x_i\partial_i + t^2 & & &\text{for $1\leq i\leq n$,}\\
x_it=tx_i, & \partial_it=t\partial_i  & &\text{for $1\leq i\leq n$.}
\end{array}$$
The Rees algebra of $A_n(K)$ is known as the homogenized Weyl algebra,
cf.~\cite{SST}.
\item Let $\frak g$ be a Lie algebra over $K$ with basis
$\{x_1,\dots,x_N\}$. The Rees algebra of the universal
enveloping algebra $U(\frak g)$ of $\frak g$ is the graded $K$-algebra  generated by
$x_1,\dots,x_N,t$ subject to the homogeneous relations
$$
\begin{array}{ll}
x_jx_i = x_ix_j + [x_j,x_i]_{\frak g}\cdot t & \text{for $1\leq i<j\leq N$,} \\
x_it=tx_i & \text{for $1\leq i\leq N$.}
\end{array}$$
This algebra is called the homogenized
enveloping algebra of $\frak g$ in \cite{Smith}.
\end{enumerate}
\end{examples}

The elements $y_i=\gr x_i\in(\gr R)_{(1)}$ generate the $K$-algebra $\gr R$. Moreover:

\begin{cor} \label{gr is homogeneous}
The associated graded algebra $\gr R$ of $R$ is homogeneous of solvable type with respect to the given monomial ordering $\leq$ of $\N^N$ and the tuple $y=(y_1,\dots,y_N)$. Moreover, if  $\deg P_{ij}<2$ for $1\leq i<j\leq N$ then $\gr R$ is semi-commutative, and $\gr R$ is commutative if and only if $\deg P_{ij}<2$ and $c_{ij}=1$ for $1\leq i<j\leq N$.
\end{cor}
\begin{proof}
The first statement follows from Lemmas~\ref{mod t},~(2) and \ref{Rees}. 
Suppose  $\deg P_{ij}<2$ for $1\leq i<j\leq N$. Then $x_jx_i=c_{ij} x_ix_j+p_{ij}$ where $p_{ij}\in R_{(<2)}$, and hence $y_jy_i=c_{ij} y_iy_j$ in $\gr R$, for $1\leq i<j\leq N$. Therefore $\gr R$ is semi-commutative, and commutative if and only if $c_{ij}=1$ for $1\leq i<j\leq N$.
\end{proof}

In each of the examples in \ref{Rees-Examples}, the associated graded algebra
is commutative. We have only considered the homogenization of $R$ with respect to the standard filtration of $R$; for other types of homogenizations see \cite[Section~4.7]{BGV}.

\medskip

Now assume that $R$ is homogeneous.
Then $R$ is a graded $K$-algebra, equipped with the grading
induced from $K\<X\>$ by $\pi\colon K\<X\>\to R$. By Lemma~\ref{R_leq d} we have
$$R_{(d)} = \bigoplus_{\abs{\alpha}=d} Kx^\alpha$$
for every $d$. Hence if $N>0$ then
\begin{equation}\label{Basic Hilbert function}
\dim_K R_{(d)} = {N+d-1\choose d} \qquad\text{for every $d$.}
\end{equation}
For a homogeneous $K$-linear subspace $V$ of $R$,
the {\bf Hilbert function $H_V\colon\N\to\N$ of $V$} is
defined by
$$H_V(d) := \dim_K V_{(d)} \qquad\text{for each $d$.}$$
Clearly if a homogeneous $K$-linear subspace $V$ of $R$ can be decomposed as a direct sum
$$V=\bigoplus_{i\in I} V_i$$
of a family $\{V_i\}_{i\in I}$ of homogeneous $K$-linear subspaces  $V_i\subseteq V$ of $R$, then
$$H_V(d)=\sum_{i\in I} H_{V_i}(d)\qquad \text{for each $d$,}$$
where all but finitely many summands in the sum on the right hand side are zero.
In many interesting cases, $H_V(d)$ agrees with a polynomial function for sufficiently large values of $d$. (Lemma~\ref{Hilbert polynomials}.) The (necessarily unique) polynomial
$P\in\Q[T]$ such that $H_V(d)=P(d)$ for all sufficiently large $d$ will be denoted by $P_V$, and called the  {\bf Hilbert polynomial of $V$.} The smallest $r\in\N$ such that $H_V(d)=P_V(d)$ for all $d\geq r$ is called the {\bf regularity} of the Hilbert function $H_V$, which we denote here by $\sigma(V)$.
For example, if $N>0$ then
$$P_R = \frac{1}{(N-1)!} (T+N-1)\cdot (T+N-2) \cdots (T+1)$$
by \eqref{Basic Hilbert function}, with $\sigma(R)=0$.
In a similar vein,
for a finitely generated graded left $R$-module $M$, each of the homogeneous components $M_{(d)}$ has finite dimension as a $K$-linear space, and the function $H_{M}\colon\N\to\N$ defined by
$$H_{M}(d) := \dim_K M_{(d)} \qquad\text{for each $d$}$$
is called the {\bf Hilbert function} of $M$. There exists a polynomial $P_M\in\Q[T]$ of degree less than $N$ with $H_M(d)=P_M(d)$ for $d$ sufficiently large, called the {\bf Hilbert polynomial of~$R$}.  The degree of $P_{M}$ is one less than the Gelfand-Kirillov dimension of the graded left $R$-module~$M$. (See, e.g., \cite[Ch.~7]{BGV}.) In particular, if $I$ is a homogeneous left ideal of~$R$, then $P_I$ exists and has degree less than $N$, and $P_{R/I}=P_R-P_I$ (if $R/I$ is considered as a left $R$-module).
We define the {\bf regularity} $r(M)$ of $H_M$ similarly to the regularity of $H_V$ above.

\section{Gr\"obner bases in Algebras of Solvable Type}
\label{Grobner bases in Algebras of Solvable Type}

In this section we let
$R=K\<x\>$ be a $K$-algebra of solvable type
with respect to a fixed monomial ordering $\leq$ of $\N^N$ and a tuple
$x=(x_1,\dots,x_N)\in R^N$.

\subsection{Left reduction}
Given $f,f',g\in R$, $g\neq 0$, we write $f\red{g} f'$ if there exist $c\in K$  and multi-indices $\alpha$, $\beta$ such that
$$\lm(x^\beta g)=x^\alpha\in\supp f, \quad
\lc(cx^\beta g)=f_\alpha,\quad
f'=f-cx^\beta g.$$
We say that $f\in R$ is {\bf reducible} by a non-zero
$g\in R$ if $\lm(g)$ divides some monomial in
the support $\supp f$ of $f$, that is, if $f\red{g}f'$ for some $f'\in R$.
In this case, if $R$ is homogeneous and $f$, $g$  are homogeneous elements of $R$, then $f'$ is also homogeneous.

Let $G$ be a subset of $R$.
We say that an element $f$ of $R$ is {\bf reducible by $G$} if $f$ is reducible by some non-zero $g\in G$;
otherwise we call $f$ {\bf irreducible} by $G$. We write $f\red{G} f'$ if $f\red{g} f'$
for some $g\in G$. The reflexive-transitive closure of
the relation $\red{G}$ is denoted by
$\overset{*}{\red{G}}$. We say that $f_0\in R$ is
a {\bf $G$-normal form} of $f\in R$ if $f\overset{*}{\red{G}} f_0$ and $f_0$ is irreducible by $G$. One may show that the relation $\red{G}$ is well-founded, hence
every element of $R$ has a $G$-normal form \cite[Lemma~3.2]{KR-W}.
If $R$ is homogeneous and $G$ consists entirely of homogeneous elements of $R$, then every homogeneous element of $R$ has a homogeneous $G$-normal form.

\subsection{Gr\"obner bases of left ideals in $R$}
Let $G$ be a finite subset of $R$. Note that if $f\overset{*}{\red{G}} f'$ ($f,f'\in R$), then there exist $g_1,\dots,g_m\in G$ and $p_1,\dots,p_m\in R$ such that
$$f=p_1g_1+\cdots+p_mg_m+f',\qquad \lm(p_1g_1),\dots,\lm(p_mg_m)\leq \lm(f).$$
In particular, if $f\overset{*}{\red{G}} 0$ then $f$ is an element of the left ideal $(G)$ of $R$ generated by $G$. If $f\overset{*}{\red{G}} 0$  for every $f\in (G)$, then
$G$ is called a {\bf Gr\"obner basis} (with respect to our
monomial ordering $\leq$).
The following proposition (for a proof of which see
\cite[Lemma~3.8]{KR-W})  gives equivalent conditions that help to identify Gr\"obner bases.

\begin{prop}\label{Groebner}
The following are equivalent:
\begin{enumerate}
\item $G$ is a Gr\"obner basis.
\item Every non-zero element of $(G)$ is reducible by $G$.
\item Every element of $R$ has a unique $G$-normal form.
\item For every non-zero $f\in (G)$ there is a non-zero $g\in G$ with
$\lm(g)|\lm(f)$.
\end{enumerate}
\end{prop}

Given a left ideal $I$ of $R$, we say that a subset $G$ of $I$ which is a Gr\"obner basis and which generates $I$ is a {\bf Gr\"obner basis of $I$} (with respect to $\leq$).
Suppose now that $G$ is a Gr\"obner basis of $I=(G)$. Given $f\in R$, we denote
by $\nf_G(f)$ the unique $G$-normal form of $f$, so $f-\nf_G(f)\in I$. Moreover, if $f,g\in R$ have distinct $G$-normal forms,  then $h:=\nf_G(f)-\nf_G(g)$ is a non-zero element of $R$ which is irreducible by $G$, so $h\notin I$ by
the equivalence of (1) and (2) in Proposition~\ref{Groebner} and thus $f-g\notin I$. Hence
two elements $f$ and $g$ of $R$ have the same $G$-normal form if and only
if $f-g\in I$.

\begin{cor}
Suppose $G$ is a Gr\"obner basis of $I$. Then
the map $$f\mapsto\nf_G(f)\colon R\to R$$ is $K$-linear, and its image
$\nf_G(R)$ satisfies
$$R=I\oplus\nf_G(R)\qquad\text{ \rom{(}internal direct sum of $K$-linear subspaces
of $R$\rom{)}.}$$
A basis of the $K$-linear space $\nf_G(R)$ is given by the set of all monomials of $R$ not divisible \textup{(}in $(x^\diamond,{\ast})$\textup{)} by
some $\lm(g)$ with $g\in G$, $g\neq 0$.
\end{cor}
\begin{proof}
Let $f,f',g\in R$, $g\neq 0$, and $c\in K$, $c\neq 0$.
If $f\red{g}f'$ then $cf \red{g} cf'$, and if $f\in R$ is $G$-irreducible,
then so is $cf$. This yields $\nf_G(cf)=c\nf_G(f)$.
Also, $h:=\nf_G(f)+\nf_G(f')$
is $G$-irreducible and $h-(f+f')\in I$, hence $h=\nf_G(h)=\nf_G(f+f')$ by the remark preceding the corollary, and
thus $\nf_G(f+f')=\nf_G(f)+\nf_G(f')$. This shows
$K$-linearity of $f\mapsto\nf_G(f)$. The rest of the corollary
is clear.
\end{proof}

Note that $\nf_G(R)$ does not depend on $G$: we have $\nf_G(R)=\nf_M(R)$ where $M$ is the $K$-linear subspace of $R$ generated by
$\lm(I)$. (Notation as introduced in Section~\ref{Monomials and Monomial Ideals}.)

Every left ideal $I$ of $R$ has a Gr\"obner basis.
(Since being a Gr\"obner basis includes being finite, this means in particular that the ring $R$ is left Noetherian.)
To see this, note that $\lm(I)$ is an ideal
of the commutative monoid of monomials of $R$ (with multiplication $\ast$). Hence there is a finite set $G$ of non-zero elements of $I$ such that
for every non-zero $f\in I$ we have $\lm(g)|\lm(f)$ for some $g\in G$;
then $G$ is a Gr\"obner basis of $I$. This argument is
non-constructive; however, as observed in \cite{KR-W}, by an adaptation of
{\bf Buchberger's algorithm} one can construct a Gr\"obner basis of $I$ from a given finite set of generators of $I$ in an effective way (up to
computations in the field $K$ and comparisons of multi-indices in $\N^N$ by the chosen monomial ordering $\leq$). The main ingredient is the following notion:

\begin{definition}
The {\bf $S$-polynomial} of elements $f$ and $g$ of $R$ is defined by
$$S (f , g) := d \lc(g) \cdot x^\alpha  f-c \lc(f) \cdot x^\beta g,$$
where $\alpha$ and $\beta$ are the unique multi-indices such that
$$x^\alpha * \lm(f) = x^\beta  * \lm(g) = \lcm\big(\lm(f), \lm(g)\big),$$
and $c = \lc(x^\alpha f)$, $d = \lc(x^\beta g)$.
\end{definition}

Now we can add the following equivalent condition (``Buchberger's criterion'') to Proposition~\ref{Groebner} (cf.~\cite[Theorem~3.11]{KR-W}):
$$\text{$G$ is a Gr\"obner basis}\quad\Longleftrightarrow\quad \text{$S(f,g) \overset{*}{\underset{G}{\longrightarrow}} 0$ for all $f,g\in G$.}$$
Starting with a finite subset $G_0$ of $R$,
Buchberger's algorithm successively constructs finite subsets $$G_0\subseteq G_1\subseteq\cdots\subseteq
G_k\subseteq\cdots$$ of elements of the left ideal $I=(G_0)$
as follows: Suppose that $G_k$ has been constructed already.
For every pair $(f,g)$ of elements of $G_k$ find a $G_k$-normal form
$r(f,g)$
of $S(f,g)$. If all of these normal forms are zero, then $G:=G_k$ is a
Gr\"obner basis of $I$, by the previous proposition, and the algorithm
terminates. Otherwise, we put $$G_{k+1}:=G_k\cup\big\{r(f,g): f,g\in G_k\big\}$$
and iterate the procedure. Dickson's Lemma guarantees that this construction eventually stops. (See \cite{KR-W} for details.)

One says that a Gr\"obner basis $G$  of the left ideal $I$ of $R$ is {\bf reduced} if  $\lc(g)=1$ and $g\in\nf_{G\setminus\{g\}}(R)$, for every $g\in G$.
Every left ideal $I$ of $R$ has a unique reduced Gr\"obner basis (see \cite[Section~4]{KR-W}); hence we can speak of {\it the}\/ reduced Gr\"obner basis of $I$. 

In summary, Gr\"obner bases of left ideals in $R$ share properties similar to Gr\"obner bases of ideals in commutative polynomial rings over $K$, with slight differences; most notably, a collection of monomials in $R$ is not automatically a Gr\"obner basis for the left ideal it generates \cite[p.~17]{KR-W}.

\subsection{Gr\"obner bases in homogeneous algebras of solvable type}
\label{Groebner homogeneous}
In this subsection $R$ is assumed to be homogeneous. From Buchberger's algorithm and earlier remarks we immediately obtain that the reduced Gr\"obner basis of each homogeneous left ideal of $R$ consists of homogeneous elements of $R$.
It is also well-known (Macaulay) that if $V$ is a homogeneous $K$-linear subspace of $R$, then
$$H_V(d) = \#\lm(V_{(d)}) \qquad\text{for every $d$.}$$
(Here and below, the cardinality of a finite set $S$ is denoted by $\#S$.)
Let now $I$ be a homogeneous left ideal of $R$ with Gr\"obner basis $G$.
The $K$-linear subspace
$M:=\nf_G(R)$ of $R$ is generated by monomials of $R$, hence is homogeneous, with $R=I\oplus M$.
Therefore, the Hilbert function of $R/I$ can be expressed as:
$$H_{R/I}(d)=H_R(d)-H_I(d)=H_M(d)=\#\lm(M_{(d)})
\quad\text{for every $d$.}$$


\subsection{Gr\"obner bases and dehomogenization}
Here we assume that $R$ is quadric (so $R^*$ is of solvable type as explained in Section~\ref{Homogenization}). We collect a few facts concerning the behavior of leading monomials, reductions, and $S$-polynomials under  dehomogenization:

\begin{lemma} \mbox{}
Let $f,f',g\in R^*$ be homogeneous, $g\neq 0$. Then
\begin{enumerate}
\item $\lm(f_*)=(\lm f)_*$, $\lc(f_*)=\lc(f)$;
\item if $f\red{g}f'$, then $f_* \red{g_*} f'_*$;
\item $\big(S(f,f')\big)_* = S(f_*,f'_*)$.
\end{enumerate}
\end{lemma}
\begin{proof}
For (1), note that $(x^\alpha t^i)_*=x^\alpha$ and $(x^\beta t^j)_*=x^\beta$, so
if $\deg(x^\alpha t^i) = \deg(x^\beta t^j)$, then
$(x^\alpha t^i)_* = (x^\beta t^j)_*$ implies $i=j$, hence
$x^\alpha t^i \leq^* x^\beta t^j$ if and only if $(x^\alpha t^i)_* \leq (x^\beta t^j)_*$. This observation immediately yields (1).
For (2), suppose  $f\red{g}f'$, and let $\alpha$, $\beta$ be multi-indices, $i,j\in\N$, and
$c\in K$ such that
$$\lm(x^\beta t^j g)=x^\alpha t^i\in\supp f, \quad
\lc(cx^\beta t^j g)=f_{(\alpha,i)}, \quad
f'=f-cx^\beta t^j g.$$
Then $(f')_* = f_* - cx^\beta g_*$, and $\lm(x^\beta g_*)=x^\alpha$ by (1).
Since $f$ is homogeneous, we have $(f_*)_\alpha = f_{(\alpha,i)}$, so
$x^\alpha\in\supp f_*$ and $\lc(cx^\beta g_*)=(f_*)_\alpha$. Thus
$f_* \red{g_*} f'_*$. For (3), let
$\alpha$, $\beta$ be multi-indices and $i,j\in\N$ such that
$$x^\alpha t^i * \lm(f) = x^\beta  t^j * \lm(f') = \lcm\big(\lm(f), \lm(f')\big),$$
and $c = \lc(x^\alpha t^i f)$, $d = \lc(x^\beta t^j f')$. Then
$$S (f , f') = d \lc(f') \cdot x^\alpha t^i f-c \lc(f) \cdot x^\beta t^j f',$$
hence
$$\big( S(f,f') \big)_* =
d\lc(f') \cdot x^\alpha f_* - c \lc(f) \cdot x^\beta f'_*.$$
By (1) we also have
$$x^\alpha * \lm(f_*) = x^\beta * \lm(f'_*) =
\lcm\big( \lm(f_*), \lm(f'_*)\big)$$
and $c=\lc(x^\alpha f_*)$, $d=\lc(x^\beta f'_*)$. This yields (3).
\end{proof}

The following corollary often allows us to reduce questions about arbitrary
Gr\"ob\-ner bases to a homogeneous situation:

\begin{cor}\label{Dehomogenization of GB}
Let $I$ be a left ideal of $R$, and let $G$ be a generating set for $I$. Let $J$ be the left ideal of $R^*$ generated by all $g^*$ with $g\in G$, and let $H$ be a Gr\"obner basis of $J$ with respect to $\leq^*$ consisting of homogeneous elements of $R^*$. Then $H_* = \{ h_* : h\in H\}$ is a Gr\"obner basis of $I$ with respect to $\leq$.
\end{cor}
\begin{proof}
We have $I=J_*=(H)_*=(H_*)$, and by parts (2) and (3) of the previous lemma
$S(f,g)\overset{*}{\red{H_*}} 0$ for all $f,g\in H_*$. Hence $H_*$ is a
Gr\"obner basis of $I$.
\end{proof}

\begin{remark}
In the situation of the previous corollary, if $H$ is reduced, then $H_*$ is not necessarily reduced. For example, suppose $R=K[x]$, the commutative polynomial ring in a single indeterminate $x$ over $K$, and $G=\{x^2,x+x^2\}$. Then $R^*=K[x,t]$ where $t$ is an indeterminate distinct from $x$, and $J=(x^2,xt+x^2)=(xt, x^2)$. So $H=\{xt,x^2\}$ is the reduced Gr\"obner basis of $J$; but $H_*=\{x,x^2\}$ is not reduced.
\end{remark}

\subsection{Gr\"obner bases and the associated graded algebra}\label{Grobner bases and the associated graded algebra}

Our algebra $R$ of solvable type comes equipped with two multi-filtrations: the standard filtration on the one hand, and the ``fine multi-filtration'' defined in Lemma~\ref{gr R is semi-commutative} on the other. In both cases, under mild assumptions, $\gr R$ is an ordinary commutative polynomial ring over $K$. (Lemma~\ref{gr R is semi-commutative} and Corollary~\ref{gr is homogeneous}.) Thus it might be tempting to try and deduce Theorem~\ref{Main Theorem} from the main result of \cite{Dube} using ``filtered-graded transfer''.
Indeed, the following is proved in \cite{Li-Wu}:

\begin{prop} \label{Li-Wu}
Suppose $\leq$ is degree-compatible. Let $I$ be a left ideal of $R$. If $G$ is a Gr\"obner basis of $I$, then $$\gr G:=\{\gr g:0\neq g\in G\}$$ is a Gr\"obner basis of the left ideal $\gr I$ of $\gr R$ consisting of homogeneous elements. Conversely, if $H$ is a Gr\"obner basis of $\gr I$ consisting of homogeneous elements and $G$ is a finite subset of $I$ with $\gr G=H$, then $G$ is a Gr\"obner basis of $I$.
\end{prop}

Proposition~\ref{Li-Wu} breaks down if $\leq$ is not degree-compatible:

\begin{example}
Suppose $R=K[x,y]$ is the commutative polynomial ring in two indeterminates $x$ and $y$ over $K$, and consider the ideal $I=(f_1,f_2,f_3)$ of $R$, where $$f_1=xy,\quad f_2=x-y^2,\quad f_3=x^2.$$
Then $G=\{f_1,f_2,f_3\}$ is not a Gr\"obner basis of $I$ with respect to the lexicographic ordering of $\N^2$ (so $y^n<x$ for every $n$), since $S(f_1,f_2)=xy-y(x-y^2)=y^3$
is irreducible by $G$. However, $\gr G$ is a Gr\"obner basis of $\gr I$  with respect to the degree-lexicographic ordering of $\N^2$. (To see this use Proposition~\ref{Li-Wu} and verify that $G$ is a Gr\"obner basis with respect to this ordering.)
\end{example}

Nevertheless, this proposition does seem to offer an easy way towards Theorem~\ref{Main Theorem} in the special case where $\leq$ is degree-compatible and $\gr R$ is commutative.  In this case we have $\gr R=K[y_1,\dots,y_N]$ where $y_i=\gr x_i$ for $i=1,\dots,N$. Unfortunately, however, if  the non-zero elements $f_1,\dots,f_n$ of $R$ generate a left ideal $I$ of $R$, then $\gr f_1,\dots,\gr f_n$ in general do not generate $\gr I$, as the following example from \cite{Li-Wu} shows:

\begin{example} Suppose $R=A_2(K)$ is the second Weyl algebra, and $I=(f_1,f_2)$ where $$f_1=x_1\partial_1,\quad f_2=x_2(\partial_1)^2-\partial_1.$$ Then $\gr f_1=\gr x_1\partial_1$, $\gr f_2=\gr x_2 \gr(\partial_1)^2$ do not generate $\gr I$. In fact, $\{\partial_1\}$ is a Gr\"obner basis for $I$ with respect to the degree-lexicographic ordering of $\N^4$.
\end{example}


It seems even less likely to be able to reduce the proof of Theorem~\ref{Main Theorem} to the associated graded algebra $\gr_\leq R$ of $R$ equipped with the fine multi-filtration. (For example, if the $K$-algebra $\gr_\leq R$ is commutative, then $\gr_\leq I$ is simply a monomial ideal of $\gr_\leq R$ in the usual sense of the word.)

\subsection{Decomposition of left ideals}

Let $I$ be a left ideal of $R$. For $f\in R$ we put
$$(I:f) := \big\{g\in R : gf \in I \big\},$$
a left ideal of $R$. If $R$, $f$ and the left ideal $I$ are homogeneous, then so is the left ideal $(I:f)$ of $R$. For $f_1,f_2\in R$ we also write $(f_1:f_2):=((f_1):f_2)$.

\begin{lemma}\label{Ideal decomposition}
Let $f\in R$, and let $G$ be a Gr\"obner basis of $(I:f)$. Then
$$I+(f) = I \oplus \nf_G(R)f.$$
\end{lemma}
\begin{proof}
Let $h \in I+(f)$. Then we can write  $h = a + bf$ with $a \in I$ and $b \in R$. Let  $c := \nf_G(b)$; then $b - c\in (I:f)$ and
$h = \big(a+(b-c)f\big) + cf$,
where the first summand is in $I$ and the second in $\nf_G(R)f$. This shows
$I+(f) = I + \nf_G(R)f$; moreover, clearly $I\cap \nf_G(R)f=\{0\}$ by construction.
\end{proof}

The previous lemma leads to a decomposition of $I$ into $K$-linear subspaces of the form $S=\nf_G(R)f$ for certain $f\in R$ and Gr\"obner bases $G$ as follows: Take $f_1,\dots,f_n\in R$, $n>0$, such that
$I=(f_1,\dots,f_n)$, and for $i=2,\dots,n$ let $G_i$ be
a Gr\"obner basis of $\big((f_1,\dots,f_{i-1}):f_i\big)$; then
$$I = (f_1)\oplus \nf_{G_2}(R)f_2\oplus\cdots\oplus \nf_{G_n}(R)f_n.$$

\begin{example}
Suppose $R=A_1(K)$ is the first Weyl algebra, so $R=K\langle x,\partial\rangle$ with the relation $\partial x - x\partial = 1$, and let $I=(f_1,f_2)$ where $f_1 = \partial$ and $f_2 = x$.
Then in fact $I=R$, and the above decomposition procedure yields
$$R = (f_1) \oplus \nf_{G_2}(R)f_2 = (\partial) \oplus K\partial\cdot x \oplus K[x]\cdot x. $$
Indeed, it is not hard to check that $G_2=\{\partial^2, x\partial-1\}$ is the reduced Gr\"obner basis of the left ideal
$(f_1:f_2)$ of $R$, with
$\nf_{G_2}(R) = K\partial \oplus K[x]$. In particular $\partial\notin (f_1:f_2)$; this is slightly
counterintuitive, since it is always true that $(I:f) \supseteq I$ in the commutative world.
\end{example}


\section{Cones and Cone Decompositions}

We first summarize the algorithmic core of Dub\'e's approach dealing with cone decompositions of monomial ideals. Afterwards, we show how to define and construct cone decompositions of homogeneous left ideals. Here, we have to adapt Dub\'e's ideas to deal with non-commutativity. We only give proofs selectively, and refer to \cite{Dube} for details.

\subsection{Monomial cone decompositions}
In this subsection we let $R$ be a $K$-linear space and $\{x^\alpha\}_\alpha$ be a monomial basis of $R$. Let $M$ be a $K$-linear subspace of $R$ spanned by monomials, and let $\cal D$
be a finite set of pairs $(w,y)$ where $w$ is a monomial in $x^\diamond$ and $y$ is a subset of $x$. We define the {\bf degree of $\cal D$} as
$$\deg\cal D := \max\big\{ \deg w  : (w,y)\in\cal D\big\}\in\N\cup\{-\infty\},$$ where $\max\emptyset=-\infty$ by convention. We also set
$$\cal D^+:=\big\{(w,y)\in\cal D : y\neq\emptyset\big\}.$$
We call $\cal D$ a {\bf cone decomposition of $M$} if
$C(w,y) \subseteq M$ for every $(w,y)\in\cal D$ and
$$M=\bigoplus_{(w,y)\in\cal D} C(w,y),$$
and $\cal D$ is a {\bf monomial cone decomposition} if $\cal D$ is a cone decomposition of some $K$-linear subspace of $R$. In the literature, ``monomial cone decompositions'' of finitely generated commutative graded $K$-algebras
are also known as ``Stanley decompositions'' (since they were first introduced   by Stanley in \cite{Stanley}). In this paper we stay with the perhaps more descriptive terminology introduced by Dub\'e in \cite{Dube}.

\begin{lemma} \label{lemma: minimal monomial}
Suppose $\cal D$ is a monomial cone decomposition of a monomial ideal $I$.  Then for each element $w$ of the minimal set of generators of $I$ there is some $y$ with $(w,y)\in\cal D$.
\end{lemma}
\begin{proof}
Since $\cal D$ is a monomial cone decomposition of $I$, there is
some $(w',y)\in\cal D$ with $w\in C(w',y)$, so $w=w'\ast a$ for some $a\in y^\diamond$.
Since $w'\in I$, we can also write $w'=w''\ast b$ for some  $w''\in F$
and $b\in x^\diamond$. So $w=w'\ast a=w''\ast b\ast a$, hence $b\ast a=1$ due to minimality of $w$, and $w=w'=w''$.
\end{proof}

In \cite{SW, MS}, algorithms are given which, upon input of a finite list of generators of a monomial ideal $I$ of $R$, produce a monomial cone decomposition for the natural complement $\nf_I(R)$ of $I$ in $R$. In fact, Dub\'e specified an algorithm which does much more, as we describe next. As before, $M$ is a $K$-linear subspace of $R$ generated by monomials, and $I$ is a monomial ideal of $R$.

\begin{definition}
We say that a pair of monomial cone decompositions $(\cal P,\cal Q)$ {\bf splits $M$ relative to $I$} if
\begin{enumerate}
\item $\cal P\cup\cal Q$ is a
cone decomposition of $M$,
\item $C(w,y)\subseteq I$ for all $(w,y)\in\cal P$,
\item $C(w,y)\cap I = \{0\}$ for all $(w,y)\in\cal Q$.
\end{enumerate}
\end{definition}

It is easy to see that if $(\cal P,\cal Q)$ is a pair of monomial cone decompositions which splits $M$ relative to $I$, then $\cal P$ is a monomial cone decomposition of $M\cap I$ and $\cal Q$ is a monomial cone decomposition of $\nf_I(M)$.

Algorithm~\ref{Splitting} accomplishes a basic task: it gives a procedure
for splitting a monomial cone relative to $I$.
The computation of a generating set $F_1$ for the monomial ideal $$(I:w\ast x_i) = ((I:w):x_i)$$ in this algorithm is carried out by Algorithm~\ref{Quotient}: if the monomial ideal $I$ is generated by $v_1,\dots,v_n\in x^\diamond$, then $(I:x_i)$ is generated by $w_1,\dots,w_n$ where
$$w_j=
\begin{cases}
v_j & \text{if $x_i$ does not divide $v_j$,} \\
 w_j=v_j/x_i & \text{otherwise,}
 \end{cases}$$
where $v_j/x_i$ denotes the monomial in $x^\diamond$ satisfying $v_j=(v_j/x_i)\ast x_i$.

\begin{algorithm}
\KwIn{$w\in x^\diamond$,  $y\subseteq x$, and a finite set $F$
of generators for $(I:w)$\;}
\KwOut{${\tt SPLIT}(w,y,F)=(\cal P,\cal Q)$, where $(\cal P,\cal Q)$ splits the monomial cone $C(w,y)$ relative to the monomial ideal $I$ of $R$\;}
\BlankLine
\lIf{$1\in F$}{\Return $\big(\{(w,y)\},\emptyset\big)$}\;
\lIf{$F\cap y^\diamond=\emptyset$}{\Return $\big(\emptyset,\{(w,y)\}\big)$}\;
\Else{choose $z\subseteq y$ maximal such that $F\cap z^\diamond=\emptyset$\;
choose $i\in\{1,\dots,N\}$ such that $x_i\in y\setminus z$\;
\BlankLine
$(\cal P_0,\cal Q_0) := {\tt SPLIT}(w,y\setminus\{x_i\},F)$;\hfill(*)\\
\BlankLine
$F_1:={\tt QUOTIENT}(F,x_i)$\;
$(\cal P_1,\cal Q_1) := {\tt SPLIT}(w\ast x_i,y,F_1)$;\hfill(**)\\
\BlankLine
\Return $(\cal P_0\cup\cal P_1,\cal Q_0\cup\cal Q_1)$\;
}
\BlankLine
\caption{Splitting a monomial cone relative to $I$.}\label{Splitting}
\end{algorithm}

\begin{algorithm}
\KwIn{a finite set $F$
of generators for a monomial ideal $I$ of $R$, and $i\in\{1,\dots,N\}$\;}
\KwOut{${\tt QUOTIENT}(F,x_i)=F'$, where $F'$ is a finite set of generators of the monomial ideal $(I:x_i)$ of $R$\;}
\BlankLine

$F':=\emptyset$\;
\While{$F\neq\emptyset$}{
choose $v\in F$\;
\lIf{$x_i|v$}{$F':=F'\cup \{v/x_i\}$\;}\Else{$F':=F'\cup\{v\}$\;}
$F:=F\setminus\{v\}$\;
}
\caption{Computing a a set of generators for $(I:x_i)$.}\label{Quotient}

\end{algorithm}

Let $w\in x^\diamond$,  $y\subseteq x$, and $F$ be a set of generators for $(I:w)$. One checks:

\begin{lemma} \

\begin{enumerate}
\item $C(w,y)\subseteq I \Longleftrightarrow 1\in F$;
\item $C(w,y)\cap I=\{0\} \Longleftrightarrow F\cap y^\diamond=\emptyset$.
\end{enumerate}
\end{lemma}

Algorithm~\ref{Splitting} proceeds by recursively decomposing the cone $C(w,y)$ as
$$C(w,y) = C(w, y\setminus\{x_i\}) \oplus C(w\ast x_i,y) \qquad (x_i\in y).$$
The lemma above shows that the base case is handled correctly.
We refer to \cite[Lemmas~4.3 and 4.4]{Dube}  for a detailed proof of the termination and correctness of Algorithm~\ref{Splitting}. The output of Algorithm~\ref{Splitting} has a convenient property:

\begin{definition} \label{def: d-standard, monomial}
We say that a monomial cone decomposition $\cal D$ is {\bf $d$-standard} if
\begin{enumerate}
\item $\deg(w) \geq d$ for all $(w,y)\in\cal D^+$;
\item for every $(w,y)\in\cal D^+$ and $d'$ with $d\leq d'\leq\deg(w)$
there is some $(w',y')\in\cal D^+$ with $\deg(w')=d'$ and $\#{y'}\geq\#{y}$.
\end{enumerate}
\end{definition}

\begin{prop} \label{standard-prop}
Let $(\cal P,\cal Q)={\tt SPLIT}(w,y,F)$. Then $\cal Q$ is $\deg(w)$-standard.
\end{prop}

In the proof of this proposition we use the following lemma:

\begin{lemma} \label{Q-lemma}
Let $(\cal P,\cal Q)={\tt SPLIT}(w,y,F)$.
\begin{enumerate}
\item For every $(v',y')\in\cal Q$ we have $F\cap (y')^\diamond=\emptyset$ and $y'\subseteq y$.
\item
For every $y'\subseteq y$ with $F\cap (y')^\diamond=\emptyset$ there exists $y''\subseteq y$ with $(w,y'')\in\cal Q$ and $\#{y''}\geq\#{y'}$.
\end{enumerate}
\end{lemma}
\begin{proof}
We prove part (1) by induction on the number of recursive calls in Algorithm~\ref{Splitting} needed to compute $(\cal P,\cal Q)$. The base case (no recursive calls) is obvious. If $(v',y')\in\cal Q_0$, then $F\cap (y')^\diamond=\emptyset$ and $y'\subseteq y\setminus\{x_i\}\subseteq y$ follows by inductive hypothesis.
Suppose $(v',y')\in\cal Q_1$; then by inductive hypothesis we obtain $F_1\cap (y')^\diamond=\emptyset$ and $y'\subseteq y$. By the way that $F_1$ is computed from $F$ in Algorithm~\ref{Quotient}, every element of $F$ is divisible by some element of $F_1$; hence $F\cap (y')^\diamond=\emptyset$.

We show part (2) by induction on $\#y-\#y'$. If $y'=y$, then the algorithm returns $\cal Q=\{(w,y)\}$, satisfying the condition in (2). Otherwise, we have $\#z\geq \#y'$ by maximality of $z$. Hence by inductive hypothesis applied to
$(\cal P_0,\cal Q_0) = {\tt SPLIT}(w,y\setminus\{x_i\},F)$, there exists $y''\subseteq y\setminus\{x_i\}$ such that $(w,y')\in\cal Q_0$ and $\#{y''}\geq\#z$.
\end{proof}

\begin{proof}[Proof of Proposition~\ref{standard-prop}] We proceed  by the number of recursions in Algorithm~\ref{Splitting} needed to compute $(\cal P,\cal Q)$. If $\cal Q$ is empty or a singleton, then the conclusion of the proposition holds trivially.
Inductively, assume that $\cal Q_0$ is $\deg(w)$-standard and $\cal Q_1$ is $(\deg(w)+1)$-standard. Let $(v',y')\in\cal Q^+$ and $d$ with $\deg(w)\leq d\leq \deg(v')$ be given; we need to show that there exists a pair $(v'',y'')\in\cal Q$ with $\deg(v'')=d$ and $\# y''\geq \# y'$. This is clear by inductive hypothesis if $(v',y')\in\cal Q_0$ or if $d\geq \deg(w)+1$.
By Lemma~\ref{Q-lemma} there exists $y''\subseteq y$ with $(w,y'')\in\cal Q$ and $\#{y''}\geq\#{y'}$, covering the case that $d=\deg(w)$.
\end{proof}

Applied to $w=1$, $y=x$, and $F=$ a set of generators for $I$, Algorithm~\ref{Splitting} produces a pair $(\cal P,\cal Q)$ consisting of a monomial cone decomposition $\cal P$ of $I$ and a monomial cone decomposition $\cal Q$ of $\nf_I(R)$.
We now analyze this situation in more detail. So suppose $I\neq R$, let $F$ be a set of generators of $I$,  and let $(\cal P,\cal Q)={\tt SPLIT}(1,x,F)$. Let also $F_{\min}\subseteq F$ be the minimal set of generators for $I$. Then:

\begin{lemma}
For every $v\in F_{\min}$ there is $(v',y')\in\cal Q$ with $\deg(v')=\deg(v)-1$.
\end{lemma}
\begin{proof}
Let $v\in F_{\min}$. By Lemma~\ref{lemma: minimal monomial} we have $(v,y) \in \cal P$ for some $y\subseteq x$. Since $1\notin F$, the pair $(v,y)$ arrived in $\cal P$ during the computation of ${\tt SPLIT}(1,x,F)$ by means of a recursive call of the form ${\tt SPLIT}(v,y,F')$ where $F'$ is a set of generators for $(I:v)$. We have $v\in I$, and thus $1\in F'$.
This shows that the recursive call must have been made in (**), because the parameter $F$ is passed on unchanged by the recursive call in (*). The call (**) occurred during the computation of some ${\tt SPLIT}(v',y,F'')$
where
$v'$ satisfies $v=v'\ast x_i$ for some $i$, and $F''$ is a
finite set of generators for $(I:v')$. Part (2) of Lemma~\ref{Q-lemma} now yields the existence of $y'\subseteq y$ such that
$(v',y')\in\cal Q$.
\end{proof}


\begin{cor}\label{Generators with bounded degree}
The set
of all $w\in F$ with $\deg(w)\leq 1+\deg(\cal Q)$ generates $I$.
\end{cor}

\begin{remark}
In \cite{MS} one finds an algorithm which, given a finite list $F$ of generators for a monomial ideal $I$ of $R$, computes a {\it Stanley filtration,}\/  that is, a list of pairs $$\big((w(1),y(1)),\dots,(w(m),y(m))\big),$$ each consisting of a monomial $w(j)$ and a subset $y(j)$ of $x$, such that for each $j$ the set $$\big\{ (w(1),y(1)), \dots, (w(j),y(j))\big\}$$ is a cone decomposition of $\nf_{I(j)}(R)$ where $$I(j):=I+ C\big(w(j+1),x\big)+ \cdots + C\big(w(m),x\big).$$
It is easy to see (since Algorithm~\ref{Splitting} and Algorithm~3.4 in \cite{MS} pursue similar ``divide and conquer'' strategies) that, for $(\cal P,\cal Q)={\tt SPLIT}(1,x,F)$, the pairs in $\cal Q$ can be ordered to form a Stanley filtration.
\end{remark}

\subsection{Cone decompositions of homogeneous ideals}
In the rest of this section, we let $R$ be a $K$-algebra of solvable type with respect to $x=(x_1,\dots,x_N)\in R^N$ and a fixed monomial ordering $\leq$ of $\N^N$. Note that in general (unless $R$ is commutative), a monomial ideal of $R$ is {\it not}\/ a left ideal of the algebra $R$.
Let $I$ be a proper left ideal of $R$; then the $K$-linear subspace $M$ of $R$ generated by $\lm(I)$ is a monomial ideal of $R$. Moreover, let  $G$ be a Gr\"obner basis of $I$; then $\lm(I)$ is
generated by $\lm(G)$, and $\nf_M(R)=\nf_G(R)$.
The central outcome of the discussion in the previous subsection is:

\begin{theorem}\label{Splitting Theorem}
The homogeneous $K$-linear subspace $\nf_G(R)$ of $R$
has a standard monomial cone decomposition. More precisely, let
$(\cal P,\cal Q)={\tt SPLIT}(1,x,F)$ where $F=\lm(G)$. Then $\cal Q$ is
a standard monomial cone decomposition of $\nf_G(R)$. Moreover,
the set of all $g\in G$ with $\deg(g)\leq 1+\deg\cal Q$
is still a Gr\"obner basis of $I=(G)$.
\end{theorem}

In this subsection we establish an analogous decomposition result (Corollary~\ref{Splitting Corollary} below) for $I$ in place of $\nf_G(R)$, provided $R$ and $I$ are homogeneous; thus: {\it until the end of this section we assume that $R$ is homogeneous.}\/ We first need to define the type of cones used in our decompositions: A {\bf cone} of $R$ is
defined by a triple $(w,y,h)$, where $w\in
x^\diamond$, $y\subseteq x$, and $h\in R$ is homogeneous:
$$C(w,y,h):=C(w,y)h = \big\{ gh : g\in C(w,y) \} \subseteq R. $$
Both monomial and general cones are homogeneous $K$-linear subspaces of $R$, and a monomial cone is a special case of a cone: $C(w,y) = C(w,y,1)$. Note, however, that $C(1,y,w) \neq C(w,y)$ in general. We introduced this definition of cone in order to be able to speak about cone decompositions of (not necessarily monomial) ideals in the non-commutative setting.

Let $M$ be a homogeneous $K$-linear subspace of $R$, and let $\cal D$
be a finite set of triples $(w,y,h)$ where $w$ a monomial in $x^\diamond$, $y$ is a subset of $x$, and $h$ is a non-zero homogeneous element of $R$. We define the {\bf degree of $\cal D$} as
$$\deg\cal D := \max\big\{ \deg(w) + \deg(h) : (w,y,h)\in\cal D\big\}\in\N\cup\{-\infty\},$$ where $\max\emptyset=-\infty$
by convention. We also set
$$\cal D^+:=\big\{(w,y,h)\in\cal D : y\neq\emptyset\big\}.$$
We call $\cal D$ a {\bf cone decomposition of $M$} if
$C(w,y,h) \subseteq M$ for every $(w,y,h)\in\cal D$ and
$$M=\bigoplus_{(w,y,h)\in\cal D} C(w,y,h).$$
and $\cal D$ is simply a {\bf cone decomposition} if $\cal D$ is a cone decomposition of some homogeneous $K$-linear subspace of $R$. By abuse of language we will also say that a cone decomposition $\cal D$ is {\bf monomial} if $h=1$ for all $(w,y,h)\in\cal D$.

\begin{lemma} \label{Hilbert polynomials}
Let $M$ be a homogeneous $K$-linear subspace $M$ of $R$ which admits a cone decomposition $\cal D$. Then the Hilbert polynomial $P_M$ of $M$ exists. In fact, for $d>\deg(\cal D)$:
$$H_M(d) = \sum_{(w,y,h)\in\cal D^+} \binom{d-\deg(w)-\deg(h)+\#y-1}{\#y-1} = P_M(d).$$
\end{lemma}
\begin{proof}
Let $h\in R$ be non-zero and homogeneous, and $w\in x^\diamond$. Then
$$H_{C(w,\emptyset,h)}(d) = \begin{cases}
0 & \text{if $d\neq\deg(w)+\deg(h)$,} \\
1 & \text{if $d=\deg(w)+\deg(h)$,}
\end{cases}$$
and for non-empty $y\subseteq x$:
$$H_{C(w,y,h)}(d) = \begin{cases}
0 & \text{if $d<\deg(w)+\deg(h)$,} \\
\binom{d-\deg(w)-\deg(h)+\#y-1}{\#y-1} & \text{if $d\geq\deg(w)+\deg(h)$.}
\end{cases}$$
Moreover, for every $d$ we have
$$H_M(d) = \sum_{(w,y,h)\in\cal D} H_{C(w,y,h)}(d).$$
The lemma now follows.
\end{proof}

In particular, if $\cal D$ is a cone decomposition of a homogeneous $K$-linear subspace $M$ of $R$, then the regularity $\sigma(M)$ of the Hilbert function of $M$ (as defined in Section~\ref{Homogenization}) is bounded by $\deg(\cal D)+1$, and for $d\geq\deg(\cal D^+)$ we have
$$H_M(d)=P_M(d) + \#\big\{ (w,y,h)\in\cal D\setminus\cal D^+ : \deg(w)+\deg(h)=d \big\}.$$
The following is an adaptation of Definition~\ref{def: d-standard, monomial}:

\begin{definition} \label{def: d-standard}
We say that a cone decomposition $\cal D$ is {\bf $d$-standard} if
\begin{enumerate}
\item $\deg(w)+\deg(h) \geq d$ for all $(w,y,h)\in\cal D^+$;
\item for every $(w,y,h)\in\cal D^+$ and $d'$ with $d\leq d'\leq\deg(w)+\deg(h)$
there is some $(w',y',h')\in\cal D^+$ with $\deg(w')+\deg(h')=d'$ and $\#{y'}\geq\#{y}$.
\end{enumerate}
We also say that $\cal D$ is {\bf standard} if $\cal D$ is $0$-standard.
\end{definition}

If $\cal D^+=\emptyset$ then  $\cal D$ is $d$-standard for every $d$,
whereas if $\cal D^+\neq\emptyset$ and $\cal D$ is $d$-standard, then
necessarily
$$d=\min\big\{\deg(w)+\deg(h): \text{$(w,y,h)\in\cal D^+$ for some $y\subseteq x$}\big\}.$$
If $\cal D$ is $d$-standard for some $d$, then we let $d_{\cal D}$ denote the smallest $d$ such that $\cal D$ is $d$-standard (so $d_{\cal D}=0$ if $\cal D^+=\emptyset$).

\begin{examples}\label{Examples for standard cone decompositions}
The empty set is a standard cone decomposition of the trivial $K$-linear subspace $\{0\}$ of $R$.
If $h\in R$ is non-zero and homogeneous, and $y\subseteq x$, then $\{(1,y,h)\}$ is
a $\deg(h)$-standard cone decomposition of $C(1,y,h)$. In particular,
$\{(1,x,1)\}$ is a standard cone decomposition of $R=C(1,x)$.
\end{examples}

The following properties are straightforward:

\begin{lemma} \label{Properties of standard cone decompositions}
\mbox{}
\begin{enumerate}
\item Suppose $M_1$ and $M_2$ are homogeneous
$K$-linear subspaces of $M$ with $M=M_1\oplus M_2$, and let $\cal
E_1$, $\cal E_2$ be cone decompositions of $M_1$ respectively $M_2$.
Then $\cal E=\cal E_1\cup\cal E_2$ is a cone
decomposition of $M$. If $\cal E_1$ and $\cal E_2$ are $d$-standard,
then so is $\cal E$.
\item Suppose $\cal D$ is a $d$-standard cone
decomposition of $M$, and let $f\in R$ be non-zero homogeneous. Then
$\cal Df := \big\{(w,y,hf):(w,y,h)\in\cal D\big\}$ is a $(d+\deg f)$-standard cone decomposition of $Mf$.
\end{enumerate}
\end{lemma}

The lemma below shows how the degrees of cone decompositions of $K$-linear subspaces decomposing the $K$-linear space $R$ are linked:

\begin{lemma} \label{max-lemma}
Let $M_1$, $M_2$ be $K$-linear subspaces of $R$ with $R=M_1\oplus M_2$. For $i=1,2$, let $\cal D_i$ be a cone decomposition of $M_i$, which is $d_i$-standard for some $d_i$. Then $$\max\{\deg\cal D_1,\deg\cal D_2\}=\max\{\deg\cal D_1^+,\deg\cal D_2^+\}.$$
\end{lemma}
\begin{proof}
We have
\begin{equation}\label{H equation}
H_{M_1}(d)+H_{M_2}(d)=H_R(d)=\binom{d+N-1}{N-1}\qquad\text{for every $d$}
\end{equation}
and thus
\begin{equation}\label{P equation}
P_{M_1}+P_{M_2} = \binom{T+N-1}{N-1}.
\end{equation}
For $d\geq\max\{\deg\cal D_1^+,\deg\cal D_2^+\}$ and $i=1,2$, we have
$$H_{M_i}(d) = P_{M_i}(d) + \#\big\{ (w,y,h)\in\cal D_i\setminus \cal D_i^+ : \deg(w)+\deg(h)=d \big\}.$$
Hence, by \eqref{H equation} and \eqref{P equation}, neither $\cal D_1$ nor $\cal D_2$ contains a triple $(w,y,h)$ with $y=\emptyset$ and $\deg(w)+\deg(h)\geq \max\big\{\deg(\cal D_1^+),\deg(\cal D_2^+)\big\}$. It follows that for $i=1,2$ we have $$\deg(\cal D_i)\leq\max\big\{\deg(\cal D_i\setminus\cal D_i^+),\deg(\cal D_i^+)\big\}\leq \max\big\{ \deg(\cal D_1^+),\deg(\cal D_2^+) \big\}$$
as required.
\end{proof}

Given $w\in x^\diamond$ as well as $y\subseteq x$ and a non-zero homogeneous $h\in R$, define
$$\cal C(w,y,h) := \big\{(w,\emptyset,h)\big\}\cup \left\{ \big(w\ast x_i, y\cap
\{x_j : j\geq i\},h\big) : x_i\in y\right\}.$$
It is easy to see that $\cal C(w,y,h)$ is a $(1+\deg h)$-standard cone decomposition of  $C(w,y,h)$.

\begin{lemma}
If $M$ has a $d$-standard cone decomposition, then $M$ has
a $d'$-standard cone decomposition for every $d'\geq d$.
\end{lemma}
\begin{proof}
If $\cal D$ is a $d$-standard cone decomposition of $M$
with $\cal D^+=\emptyset$, then $\cal D$ is $d'$-standard for all $d'$.
Therefore, suppose $\cal D$ is a $d$-standard cone decomposition
of $M$ with $\cal D^+\neq\emptyset$; it is enough to show that then
$M$ has a $(d+1)$-standard cone decomposition. Now put
$$\cal E := \big\{ (w,y,h)\in\cal D: \deg(w)+\deg(h)=d \big\}.$$
Then trivially $\cal E$ is $d$-standard and, since $\cal D$ is $d$-standard,
$\cal D\setminus\cal E$ is $(d+1)$-standard. Put
$$\cal E' := \bigcup_{(w,y,h)\in\cal E} \cal C(w,y,h).$$
Then $\cal E'$ is a $(d+1)$-standard cone decomposition of 
$\bigoplus_{(w,y,h)\in\cal E} C(w,y,h)\subseteq M$. Hence
$\cal E'\cup (\cal D\setminus\cal E)$ is a $(d+1)$-standard cone decomposition of $M$.
\end{proof}

\begin{cor} \label{Splitting standard cone decompositions}
Let $M_1,\dots,M_r\subseteq M$ be homogeneous $K$-linear subspaces of $R$ with
$M=M_1\oplus\cdots\oplus M_r$. If each $M_i$ has a $d_i$-standard cone
decomposition, then $M$ has a $d$-standard cone decomposition where
$d=\max\{d_1,\dots,d_r\}$.
\end{cor}


Combining Theorem~\ref{Splitting Theorem} with Corollary~\ref{Splitting standard cone decompositions} we obtain:

\begin{cor} \label{Splitting Corollary}
Let $I=(f_1,\dots,f_n)$ be a left ideal of $R$ where $f_1,\dots,f_n\in R$ are non-zero and homogeneous, and suppose $n>0$. Let $d_i = \deg (f_i)$ for $i=1,\dots,n$, and $d=\max \{ d_1, \dots,d_n \}$.
Then there is a $K$-linear subspace $M$ of $I$ with $I=(f_1)\oplus M$, which admits a
$d$-standard cone decomposition $\cal D$. \textup{(}Hence $\{(1,x,f_1)\}\cup\cal D$ is
a $d$-standard cone decomposition of $I$.\textup{)}
\end{cor}
\begin{proof}
For $i=2,\dots,n$ let $G_i$ be a Gr\"obner basis of
$((f_1,\dots,f_{i-1}):f_i)$. Then $$I=(f_1)\oplus M \qquad\text{for $M:=\nf_{G_2}(R)f_2\oplus\cdots\oplus \nf_{G_n}(R)f_n$,}$$
as in the remark after Lemma~\ref{Ideal decomposition}.
The principal left ideal $(f_1)$ has a
$d_1$-standard cone decomposition $\{(1,x,f_1)\}$
(Example ~\ref{Examples for standard cone decompositions}). For each
$i=2,\dots,n$ let $\cal D_i$ be a standard monomial cone decomposition of
$\nf_{G_i}(R)$ guaranteed by Theorem~\ref{Splitting Theorem};
then $${\cal D}_i f_i = \big\{(w,y,f_i) : (w,y)\in {\cal D}_i\big\}$$ is a
$d_i$-standard cone decomposition of $\nf_{G_i}(R) f_i$
by Lemma~\ref{Properties of standard cone decompositions},~(2).
The claim now follows from Corollary~\ref{Splitting standard cone
decompositions}.
\end{proof}

\subsection{Macaulay constants and exact cone decompositions}

What is stated in this subsection generalizes the corresponding concepts in Section~6 of \cite{Dube}.
Let $\cal D$ be a cone decomposition which is $d$-standard for some $d$. For every $i$ we define the cone decomposition
$$\cal D_i := \big\{ (w,y,h) \in\cal D: \#y\geq i\big\}.$$
Then we have
$$\cal D=\cal D_0\supseteq\cal D^+=\cal D_1\supseteq\cdots\supseteq \cal D_N\supseteq \cal D_{N+1}=\emptyset.$$
We define the {\bf Macaulay constants} $b_0,\dots,b_{N+1}$ of $\cal D$ as follows:
$$ b_i := \max\big\{ d_{\cal D}, 1+\deg \cal D_i \big\} = \begin{cases}
d_{\cal D} & \text{if $\cal D_i=\emptyset$} \\
1+\deg\cal D_i & \text{otherwise.}
\end{cases}$$
From the definition it follows that $b_0\geq\dots\geq b_{N+1}=d_{\cal D}$. The integer $b_0$ is an upper bound for the regularity $\sigma(M)$ of $H_M$.
The name of the constants is due to the fact that Macaulay proved that if $R$ is commutative and $I$ a homogeneous ideal of $R$, then there are integers $b_0\geq\cdots\geq b_{N+1}\geq 0$ such that
$$H_{R/I}(d) = \binom{d-b_{N+1}+N}{N}-1-\sum_{i=1}^{N}\binom{d-b_i+i-1}{i}\qquad\text{for $d\geq b_0$.}$$
The $b_i$ turn out to
be the Macaulay constants of a special type of monomial cone decomposition of $\nf_G(R)$ (for an arbitrary Gr\"obner basis $G$ of $I$), which we now define in general:

\begin{definition}
A cone decomposition $\cal D$ is called {\bf exact} if $\cal D$ is $d$-standard for some $d$ and for every degree $d'$, $\cal D^+$ contains at most one triple $(w,y,h)$ with $\deg(w) + \deg(h) = d'$.
\end{definition}

Exact cone decompositions have a strong rigidity property:

\begin{lemma}\label{ECDs are rigid}
Let $\cal D$ be an exact cone decomposition with Macaulay constants $b_i$. Then for each $i=1,\dots,N$ and each $d$ with $b_{i+1}\leq d<b_i$ there is exactly one $(w,y,h)\in\cal D^+$ such that $\deg(w)+\deg(h)=d$, and for this triple we have $\#y=i$.
\end{lemma}
\begin{proof}
Suppose $d$ satisfies $b_{i+1}\leq d <b_i$. Let $(w',y',h')\in\cal D$ be such that $\# y'\geq i$ and $\deg(w')+\deg(h')=b_i-1$. Then, since $\cal D$ is $d_{\cal D}$-standard, there exists $(w,y,h)\in\cal D$ with $\deg(w)+\deg(h)=d$ and $\# y\geq \# y'\geq i$. We have $\# y=i$, since otherwise $(w,y,h)\in\cal D_{i+1}$ with $\deg(w)+\deg(h)=d\geq b_{i+1}>\deg\cal D_{i+1}$, contradicting the definition of $b_{i+1}$. By exactness of $\cal D$, $(w,y,h)$ is the only triple in $\cal D^+$ with $\deg(w)+\deg(h)=d$.
\end{proof}

The next lemma allows one to split triples in cone decompositions to achieve exactness:

\begin{lemma} \label{Exactness-split}
Let $\cal D$ be a $d$-standard cone decomposition of the $K$-linear subspace $M$ of $R$, and let $(w,y,h), (v,z,g)\in\cal D$ such that
$$\deg(w)+\deg(h)=\deg(v)+\deg(g), \qquad \# z\geq \# y>0.$$
Let $x_i\in y$ be arbitrary. Then
$$\cal D' := \big(\cal D\setminus \big\{ (w,y,h) \big\}\big) \cup \big\{ (w,y\setminus\{x_i\}, h), (w\ast x_i, y, h)\big\}$$
is also a $d$-standard cone decomposition of $M$.
\end{lemma}
\begin{proof}
We have
$$C(w,y,h)=C(w,y\setminus\{x_i\},h)\oplus C(w\ast x_i,y,h).$$
So $\cal D'$ remains a cone decomposition of $M$, and it is easy to see that $\cal D'$ is $d$-standard.
\end{proof}

By a straightforward adaptation of Algorithms~{\tt SHIFT} and {\tt EXACT} in \cite{Dube}, and using Lemma~\ref{Exactness-split} instead of Lemma~6.2 of \cite{Dube} in verifying their correctness, one obtains:

\begin{theorem}\label{Theorem-Exact}
There exists an algorithm that, given a $d$-standard cone decomposition $\cal D$ of a $K$-linear subspace $M$ of $R$, produces an exact $d$-standard decomposition $\cal D'$ of $M$,
whose Macaulay constant $b_0$ satisfies $b_0\geq 1+\deg(\cal D)$.
\end{theorem}

Let now $\cal D$ be an exact cone decomposition of a $K$-linear subspace $M$ of $R$. Then
$$P_M(T)  = \sum_{i=1}^N \sum_{j=b_{i+1}}^{b_i-1} \binom{ T-j+i-1 }{ i-1 }$$
by  Lemmas~\ref{Hilbert polynomials} and \ref{ECDs are rigid}.
One may show that this sum can be converted to
$$P_M(T) =  \binom{T-b_{N+1}+N}{N}-1-\sum_{i=1}^{N}\binom{T-b_i+i-1}{i},$$
and once $b_{N+1}=d_{\cal D}$ has been fixed, the coefficients $b_1,\dots,b_N$ uniquely determine the polynomial $P_M$; see \cite[p.~768--769]{Dube};
also, $b_0$ is the smallest $r\geq b_1$ such that $H_M(d)=P_M(d)$ for all $d\geq r$.
In particular, the Macaulay constants $b_0\geq b_1\geq\cdots\geq b_{N+1}=0$ of an exact {\it standard}\/ cone decomposition $\cal D$ of $M$ do not depend on our choice of $\cal D$, and the Hilbert function of $M$ is uniquely determined by $b_0,\dots,b_N$. Since every $K$-linear subspace $M$ which admits a standard cone decomposition also has an exact standard cone decomposition (by the previous theorem), we may, in this case, simply talk of the {\bf Macaulay constants $b_0,\dots,b_N$ of $M$}.
All this applies to $M=\nf_G(R)$ where $G$ is a Gr\"obner basis of a left ideal of $R$; hence, by Theorems~\ref{Splitting Theorem} and \ref{Theorem-Exact} we obtain:

\begin{cor} \label{cor: bound by b_0}
Let $G$ be the reduced Gr\"obner basis of a left ideal of $R$, and  let $b_0,\dots,b_N$ be the Macaulay constants of $\nf_G(R)$. Then $\deg(g) \leq b_0$ for every $g \in G$.
\end{cor}

\section{Proof of Theorem~\ref{Main Theorem} and its Corollaries}

Let $R$ be a $K$-algebra of solvable type with respect to $x=(x_1,\dots,x_N)$ and a monomial ordering $\leq$ of $\N^N$, where $N>0$.

\subsection{Degree bounds for Gr\"obner bases}
Let $I$ be a left ideal of $R$ generated by non-zero elements $f_1,\dots,f_n\in R$, where $n>0$, and let $d$ be the maximum of the degrees of $f_1,\dots,f_n$. The central result of this section is:

\begin{prop}\label{Main Prop}
Suppose the algebra $R$ and the generators $f_i$ of $I$ are homogeneous, and $N>1$. Then the elements of the reduced Gr\"obner basis of $I$ have degree at most $$D(N-1,d)=2\left(\frac{d^2}{2}+d\right)^{2^{N-2}}.$$
\end{prop}

Before we give the proof we state an estimate proved in  \cite[Section~8]{Dube}:

\begin{lemma}
Let $a_1\geq\cdots\geq a_{N}\geq d$ and $b_1\geq\cdots\geq b_{N}\geq 0$ be integers, and suppose that we have an equality of polynomials
$\binom{T+N-1}{N-1} = P(T)+Q(T)$
where
\begin{align}
P(T) &=  \binom{T-d+N}{N} + \binom{T-d+N-1}{N-1} -1
- \sum_{i=1}^N \binom{T-a_i+i-1}{i}  \label{Polynomials} \\
Q(T) &= \binom{T+N}{N} -1 -\sum_{i=1}^N \binom{T-b_i+i-1}{i}. \notag
\end{align}
Then $a_j+b_j \leq D(N-j,d)$ for $j=1,\dots,N-1$.
\end{lemma}

\begin{proof}[Proof of Proposition~\ref{Main Prop}]
After reordering the $f_1,\dots,f_n$ we may assume that $\deg(f_1)=d$. Let $G$ be the reduced Gr\"obner basis of $I$, and
let $\cal D$ be a standard exact cone decomposition of $\nf_G(R)$, with Macaulay constants $b_0\geq\cdots\geq b_{N+1}=0$. Let $\cal E$ be a $d$-standard exact cone decomposition of a $K$-linear subspace $M$ of $I$ such that $I=(f_1)\oplus M$ (by Corollary~\ref{Splitting Corollary} and Theorem~\ref{Theorem-Exact}), with Macaulay coefficients $a_0\geq\cdots\geq a_{N+1}=d$.
Then $\cal E\cup\{(1,x,f_1)\}$ is a
$d$-standard (but not exact) cone decomposition  of $I$, with the same Macaulay constants $a_0,\dots,a_{N-1}$ as $\cal E$. The Hilbert polynomials of $I$ and $\nf_G(R)$ are given by the polynomials $P$ respectively $Q$ as in \eqref{Polynomials}. Hence
$a_1+b_1\leq D:=D(N-1,d)$, so  $\max\{a_0,b_0\}=\max\{a_1,b_1\}\leq D$ by Lemma~\ref{max-lemma}.
Now apply Corollary~\ref{cor: bound by b_0}. 
\end{proof}

\begin{remark}
Suppose the hypothesis of the previous proposition holds.
Implicit in the proof above, there is the uniform bound
$$\sigma(R/I) \leq D(N-1,d)$$
for the regularity of the Hilbert function of the left $R$-module $R/I$.
A similar doubly-exponential bound for $\sigma(R/I)$ was obtained (in the case of Weyl algebras) in \cite{Chistov-Grigoriev-2007}. In the case where $R$ is a commutative polynomial ring, the regularity of the Hilbert function $\sigma(M)$ of a finitely generated $R$-module $M$ is closely related to the Castelnuovo-Mumford regularity $\reg(M)$ of $M$. For example (see \cite[2.1]{C-MS}), in this case
$$\sigma(R/I) \leq \reg(R/I)=\reg(I)-1.$$
There does exist a doubly-exponential bound on $\reg(I)$ in terms of $N$ and $d$, valid independently of the characteristic of $K$ (see \cite{CS}):
$$\reg(I)\leq (2d)^{2^{N-2}}.$$
It would be interesting to see whether this bound can also be deduced using the methods of the present paper.
\end{remark}

We next address the inhomogeneous case:

\begin{cor} \label{Detailed bound}
Suppose $R$ is quadric.
Then there exists a Gr\"obner basis $G$ of $I$ with the following property: for every $g\in G$ we can write
$$g = y_{g,1} f_1 + \cdots + y_{g,n}f_n$$
where $y_{g,i}\in R$ with $$\deg(y_{g,i}f_i)\leq D(N,d)=2\left(\frac{d^2}{2}+d\right)^{2^{N-1}}\qquad\text{for $i=1,\dots,n$,}$$
and such that for $i=1,\dots,n$ each $f_i$ can be expressed as
$$f_i = \sum_{g\in G} z_{i,g} g$$
where $z_{i,g}\in R$, all but finitely many $z_{i,g}=0$, and
$\deg(z_{i,g} g) \leq d$ for every $g\in G$.
\end{cor}
\begin{proof}
By the proposition above,  the reduced Gr\"obner basis $H$ with respect to $\leq^*$ of the left ideal of $R^*$ generated by $f_1^*,\dots,f_n^*$ consists of homogeneous elements of degree at most $D(N,d)$. Hence for every $h\in H$ there are homogeneous $y_{h,1},\dots,y_{h,n}\in R^*$ with $$h=y_{h,1}f_1^*+\cdots+y_{h,n}f_n^*$$
and $$\deg( y_{h,i}f_i^* ) \leq \deg (h) \leq D(N,d)\qquad\text{for $i=1,\dots,n$.}$$
Corollary~\ref{Dehomogenization of GB} shows that $G:=H_*$ is a Gr\"obner basis of $I$ with respect to $\leq$, and for every $h\in H$ we have $$h_*=y_{h_*,1}f_1+\cdots+y_{h_*,n}f_n$$ with $y_{h_*,i}:=(y_{h,i})_*$ and
$$\deg(y_{h_*,i}f_i) = \deg(y_{h,i}f_i^*)\leq D(n,d) \qquad\text{for $i=1,\dots,n$,}$$
as required. Similarly, each $f_i^*$ can be expressed as $f_i^*=\sum_{h\in H} z_{i,h}h$ where $z_{i,h}\in R^*$ are homogeneous and $\deg(z_{i,h}h)\leq\deg(f_i^*)\leq d$ for every $i$ and $h\in H$, and this yields the requirement on the $f_i$.
\end{proof}

The previous corollary yields Theorem~\ref{Main Theorem}.
Before we are able to compute a degree bound for reduced Gr\"obner bases which is also valid in the inhomogeneous situation, we need to study the complexity of reduction sequences.

\subsection{Degree bounds for normal forms} \label{Degree bounds for normal forms}
Here we assume $d>0$; we also
let $\omega$ be a given multi-index with positive components, and write $\wt=\wt_{\omega}$.
For non-zero $f\in R$ we set
$$\wt(f) := \max_{\alpha\in\supp(f)} \wt(\alpha),$$
and we let $\wt(0):=0$. Then for all $f,g\in R$ we have
\begin{equation}\label{weight inequality, f}
\deg(f) \leq \wt(f) \leq ||\omega||\,\deg(f)
\end{equation}
by \eqref{weight inequality}. Also
$$\wt(f+g)\leq \max\big\{\wt(f),\wt(g)\big\}, \qquad \wt(cf)=\wt(f)\ \text{for non-zero $c\in K$.}$$
From Proposition~\ref{Approximate leq by wt function} we obtain:

\begin{lemma}\label{Choice of omega}
Given $d$, one can choose $\omega$ with
$||\omega|| \leq 2d(N+1)N^{N/2}$
such that
$$\wt_\omega(f)=\wt_\omega\big(\lm(f)\big)\qquad\text{for all $f\in R$ with $\deg(f)\leq d.$}$$
\end{lemma}

We will need a variant of \cite[Lemma~1.4]{KR-W}; the proof is analogous and left to the reader. Here we assume that the commutator relations between $x_i$ and $x_j$ in $R$ are expressed as in Definition~\ref{Solvable Type}.

\begin{lemma}
Suppose $\wt(p_{ij})<\wt(x_ix_j)$ for $1\leq i<j\leq N$. Then for all $\alpha$, $\beta$ we have
$$x^\alpha\cdot x^\beta=c x^{\alpha+\beta}+r \qquad\text{where $c\in K$, $c\neq 0$, and $\wt(r)<\wt(x^{\alpha+\beta})$,}$$
in particular $\wt(x^\alpha\cdot x^\beta)=\wt(x^\alpha)+\wt(x^\beta)$.
\end{lemma}

We can now show:

\begin{lemma}\label{Reduction-Sequence}
Suppose $d$ satisfies $\deg(p_{ij})\leq d$ for $1\leq i<j\leq N$, and
let $G$ be a subset of $R$ each of whose elements has degree at most $d$.
If $f \overset{*}{\underset{G}{\longrightarrow}} h$, where $f,h\in R$, then
there are $g_1,\dots,g_m\in G$ and $p_1,\dots,p_m\in R$ with
$$f-h=p_1g_1+\cdots+p_mg_m$$
and
$$\deg(p_1g_1),\dots,\deg(p_m g_m),\deg(h) \leq \deg(f)\, 2d(N+1)N^{N/2}.$$
\end{lemma}
\begin{proof}
Choose a weight vector $\omega$ with positive components according to Lemma~\ref{Choice of omega}, and write $\wt=\wt_\omega$. In the following
we also let $g$ range over $G$. Proceeding by Noetherian induction on the well-founded relation $\underset{G}{\longrightarrow}$, by the inequalities in \eqref{weight inequality, f} it suffices to show that if
$f \overset{*}{\underset{G}{\longrightarrow}} h$, then
there are $g_1,\dots,g_m\in G$ and $p_1,\dots,p_m\in R$ with
$$f-h=p_1g_1+\cdots+p_mg_m$$
and
$$\wt(p_1g_1),\dots,\wt(p_m g_m) \leq \wt(f).$$
Suppose $f \underset{g}{\longrightarrow} f' \overset{*}{\underset{G}{\longrightarrow}} h$.
Then there exists $c\in K$  and $\alpha$, $\beta$ such that
$$\lm(x^\beta g)=x^\alpha\in\supp f, \quad
\lc(cx^\beta g)=f_\alpha,\quad
f'=f-cx^\beta g.$$
Now by the previous lemma and  the choice of $\omega$, we have
$$\wt(c x^{\beta} g) = \wt(x^\beta) + \wt(g) = \wt(x^\beta) + \wt\big(\lm(g)\big) = \wt(x^\alpha) \leq \wt(f)$$
and thus $\wt(f')\leq \wt(f)$.
By inductive hypothesis, there are $g_i\in G$ and $p_i\in R$ with
$$f'-h=p_1 g_1+\cdots+p_ng_n\quad\text{and}\quad
\wt(p_i g_i) \leq \wt(f')\text{ for every $i$.}$$
Hence
$$f-h = (f-f') + (f'-h) = p_1 g_1 +\cdots +p_n g_n+p_{n+1}g_{n+1}$$
where $p_{n+1} := c x^\beta$, $g_{n+1}:=g$ satisfy $\wt(p_i g_i)\leq \wt(f)$ for every $i$, as required.
\end{proof}

If $\leq$ is degree-compatible, then the estimate in the lemma above can be improved, and the additional assumption on $d$ removed: Let $G$ be a subset of $R$, $f,h\in R$; if
$f \overset{*}{\underset{G}{\longrightarrow}} h$, then there are $g_1,\dots,g_m\in G$ and $p_1,\dots,p_m\in R$
such that
$$f-h=p_1g_1+\cdots+p_mg_m$$
and
$$\lm(p_1g_1),\dots,\lm(p_m g_m),\lm(h)\leq \lm(f).$$
Since our monomial ordering is degree-compatible, we have
$$\deg(p_1g_1),\dots,\deg(p_m g_m),\deg(h) \leq \deg(f).$$

\subsection{Degree bounds for reduced Gr\"obner bases}
{\it In the rest of this section we assume that $R$ is quadric.}\/ The results from the previous subsection allow us to show Corollary~\ref{Corollary 1 to Main Theorem}:

\begin{cor}
The reduced Gr\"obner basis of every left ideal of $R$ generated by elements of degree at most $d$ consists of elements of degree at most
$$2\,D(N+1,d)\,(N+1)\,N^{N/2}.$$
\end{cor}
\begin{proof}
We may assume $d>0$; put $D:=D(N,d)$, so $D>2$.
Let $I$ be a left ideal of $R$ generated by elements of degree at most $d$. Choose a Gr\"obner basis $G=\{g_1,\dots,g_m\}$ of $I$ with $\deg(g_i)\leq D$ for $i=1,\dots,m$. (Corollary~\ref{Detailed bound}.) After pruning $G$ if necessary, we may assume that $\lm(G)$ is a minimal set of generators for the monomial ideal of $R$ generated by $\lm(I)$, and after normalizing each $g_i$, that $\lc(g_i)=1$ for every $i$. Set $h_i := g_i-\lm(g_i)$ for every $i$. Then by Lemma~\ref{Reduction-Sequence} we have
$$\deg \nf_G(h_i) \leq \deg(h_i)\, 2D\,(N+1)N^{N/2} \leq 2D^2\,(N+1)N^{N/2}.$$
Then $G':=\{g_1',\dots,g_m'\}$ where $g_i' := \lm(g_i)+\nf_G(h_i)$ for every $i$ is a reduced Gr\"obner basis of $I$ the degrees of whose elements $g_i'$ obey the stated bound.
\end{proof}

For degree-compatible monomial orderings one obtains in a similar way:

\begin{cor} \label{Reduced GB for deg-compatible}
Suppose that the monomial ordering $\leq$ is degree-compatible. Then the reduced Gr\"obner basis of every left ideal of $R$ generated by elements of degree at most $d$ consists of elements of degree at most $D(N,d)$.
\end{cor}

\subsection{Ideal membership}
Now we turn to
degree bounds for solutions to linear equations. In particular, we'll show Corollary~\ref{Corollary 2 to Main Theorem}.

\begin{prop} \label{Ideal membership}
If $f\in I=(f_1,\dots,f_n)$
where $f_1,\dots,f_n\in R$ are of degree at most $d$, then there there are $y_1,\dots,y_n\in R$ of degree at most
$$D(N,d)\cdot\left(2\,\deg(f)\,(N+1)\,N^{N/2}+1\right)$$
with $$f=y_1f_1+\cdots+y_nf_n.$$
\end{prop}

\begin{proof}
We may assume $d>0$; put $D:=D(N,d)$. Let $f_1,\dots,f_n\in R$ have degree at most $d$, and $f\in I$. Choose a Gr\"obner basis $G$ of $I=(f_1,\dots,f_n)$ with the property stated in Corollary~\ref{Detailed bound}. Then by Lemma~\ref{Reduction-Sequence} there are $g_1,\dots,g_m\in G$ and $p_1,\dots,p_m\in R$
with
$$f=p_1g_1+\cdots+p_mg_m$$
and
$$\deg(p_1g_1),\dots,\deg(p_m g_m)\leq \deg(f)\,2\,D\, (N+1)\, N^{N/2}.$$
Write each $g_i$ as
$$g_i=y_{i,1} f_1 + \cdots + y_{i,n} f_n$$
where $y_{i,j}\in R$ satisfies $\deg(y_{i,j}f_j)\leq D$. Then
$$f=y_1f_1+\cdots+y_nf_n$$
where each
$y_j:=  \sum_{i} p_i y_{i,j}$ satisfies the claimed degree bound.
\end{proof}

{\it In the rest of this section, we restrict ourselves to the case that the monomial ordering $\leq$ is degree-compatible.}\/
In a similar way as above we then obtain:

\begin{prop} \label{Degree-compatible}
Let $f_1,\dots,f_n\in R$ be of degree at most $d$, and $f\in R$. If $$f=y_1f_1+\cdots+y_nf_n$$ for some $y_1,\dots,y_n\in R$, there are such $y_i$ of degree at most $\deg(f)+D(N,d)$.
\end{prop}

\subsection{Generators for syzygy modules}

Below, the left $R$-module of left syzygies of a tuple $f=(f_1,\dots,f_n)\in R^n$ is denoted by $\Syz(f)$ (a submodule of the free left $R$-module $R^n$).

Suppose $G=\{g_1,\dots,g_m\}$ is a Gr\"obner basis in $R$.
For  $1\leq i<j\leq m$ let  $\alpha_{ij}$ and $\beta_{ij}$ be the unique multi-indices such that
$$x^{\alpha_{ij}}\ast\lm(g_i)  = x^{\beta_{ij}}\ast\lm(g_j) = \lcm\big( \lm(g_i),\lm(g_j) \big)$$
and
$$c_{ij}  := \lc(x^{\alpha_{ij}}g_i), \qquad d_{ij} := \lc(x^{\beta_{ij}}g_j).$$
Each $S$-polynomial
$$S(g_i,g_j) = d_{ij}\lc(g_j) x^{\alpha_{ij}}g_i - c_{ij}\lc(g_i) x^{\beta_{ij}} g_j$$
admits a representation of the form
$$S(g_i,g_j) = \sum_{k=1}^m p_{ijk} g_k, \quad
\lm(p_{ijk}g_k) \leq \lm S(g_i,g_j)  \qquad (p_{ijk}\in R).$$
Now consider the vectors
$$s_{ij}:=  d_{ij}\lc(g_j)x^{\alpha_{ij}}e_i - c_{ij}\lc(g_i)x^{\beta_{ij}}e_j -\sum_k p_{ijk} e_k \qquad (1\leq i<j\leq m)$$
in $R^m$. Here $e_1,\dots,e_m$ denotes the standard basis of the free left $R$-module $R^m$. Obviously, each $s_{ij}$ is a left syzygy of $(g_1,\dots,g_m)$; in fact (see \cite[Theorem~3.15]{KR-W}), the syzygies $s_{ij}$ generate the left $R$-module $\Syz(g_1,\dots,g_m)$.
We denote the set of $m\times n$-matrices with entries in $R$ by $R^{m\times n}$. The $n\times n$-identity matrix is denoted by $I_n$. The following transformation rule for syzygies is easy to verify:

\begin{lemma}
Let $f=(f_1,\dots,f_n)^\tr\in R^n$ and $g=(g_1,\dots,g_m)^\tr\in R^m$, and suppose $A\in R^{m\times n}$, $B\in R^{n\times m}$ such that $g=Af$ and $f=Bg$. Let $M$ be a matrix whose rows generate $\Syz(g)$. Then $\Syz(f)$ is generated by the rows of the matrix $$\left[\ \  \begin{matrix} MA \\ \hline  I_n-BA\end{matrix}\ \  \right].$$
\end{lemma}

We now use these facts in the proof of:

\begin{prop}
Let $f=(f_1,\dots,f_n)^\tr\in R^n$ be of degree at most $d$. Then $\Syz(f)$ can be generated by elements of degree at most $3D(N,d)$.
\end{prop}
\begin{proof}
Let $g=(g_1,\dots,g_m)^\tr\in R^m$ be such that $G=\{g_1,\dots,g_m\}$ is a Gr\"obner basis of the left ideal of $R$ generated by $f_1,\dots,f_n$ as in Corollary~\ref{Detailed bound}. Then there are  $A\in R^{m\times n}$ of degree at most $D(N,d)$ and $B\in R^{n\times m}$ of degree at most $d$ such that $g=Af$ and $f=Bg$. Each $S$-polynomial $S(g_i,g_j)$ has degree at most $2D(N,d)$; hence there exists a matrix $M$ of degree at most $D(N,d)$ whose rows generate $\Syz(g)$. Since
$\deg(MA)\leq 3D(N,d)$ and $\deg(AB)\leq D(N,d)+d\leq 3D(N,d)$, the claim follows from the previous lemma.
\end{proof}

\section{Two-sided Ideals}

In this section we deduce Corollary~\ref{Corollary 3 to Main Theorem} on degree bounds for two-sided ideals from the results of the previous two sections. Throughout let $R$ again be an algebra over a field $K$.

\subsection{Gr\"obner bases of two-sided ideals in $R$}
In this subsection, suppose that $R=K\<x\>$ is of solvable type with respect to $x=(x_1,\dots, x_N)$ and some monomial ordering $\leq$ of $\N^N$.
It is possible to define a notion of Gr\"obner basis for two-sided ideals of $R$:

\begin{prop} Let $G$ be a finite subset of $R$.
The following statements are equivalent:
\begin{enumerate}
\item $G$ is a Gr\"obner basis, and the two-sided ideal of $R$ generated by $G$ agrees with the left ideal $(G)$ of $R$ generated by $G$.
\item $G$ is a Gr\"obner basis, and $gx_i\in (G)$ for every $g\in G$ and $i=1,\dots,N$.
\item For every non-zero element $f$ of the two-sided ideal of $R$ generated by $G$ there exists a non-zero $g\in G$ with $\lm(g)|\lm(f)$.
\end{enumerate}
\end{prop}

If a finite subset $G$ of $R$ satisfies one of the equivalent conditions in this proposition (proved in \cite[Theorem~5.4]{KR-W}), then $G$ is called a {\bf two-sided Gr\"obner basis} (with respect to $\leq$). If $I$ is a two-sided ideal of $R$, then a subset $G$ of $I$ is called a {\bf Gr\"obner basis} of $I$ (with respect to $\leq$) if $G$ is a two-sided Gr\"obner basis which also generates the two-sided ideal $I$. The main result of this section is the following:

\begin{prop}\label{Two-sided}
Suppose $R$ is quadric.
Every two-sided ideal of $R$ generated in degree at most $d$ has a two-sided Gr\"obner basis consisting of elements of degree at most $D(2N,d)$.
\end{prop}

The proof of this proposition uses enveloping algebras, which we introduce next.

 \subsection{The enveloping algebra}\label{The opposite algebra and the enveloping algebra}
The {\bf opposite algebra} of $R$ is the $K$-algebra $R^\op$ whose underlying $K$-linear space is the same as that of $R$ and whose multiplication operation $\cdot^\op$ is given by $a\cdot^\op b=b\cdot a$ for $a,b\in R$.
The {\bf enveloping algebra} of $R$ is the $K$-algebra $R^\env:=R\otimes_K R^\op$. There is a natural one-to-one correspondence between $R$-bimodules and left $R^\env$-modules: every $R$-bimodule $M$ has a left $R^\env$-module structure given by
$$(a\otimes b)\cdot f = a f b\qquad
\text{for $a\in R$, $b\in R^\op$, and $f\in M$,}$$
and conversely, every left $R^\env$-module $M'$ also carries an $R$-bimodule structure with
$$af'b=(a\otimes b)f'\qquad \text{for $a\in R$, $b\in R^\op$, and $f'\in M'$.}$$
There is a surjective morphism $\mu\colon R^\env\to R$ of left $R^\env$-modules with $\mu(a\otimes b)=ab$ for $a\in R$, $b\in R^\op$.
For every $n$, acting component by component, $\mu$ induces a surjective morphism $(R^\env)^n\to R^n$ of left $R^\env$-modules, which we also denote by $\mu$.
Thus for every $R$-sub-bimodule $M$ of $R^n$ we obtain a left $R^\env$-submodule $\mu^{-1}(M)$ of $(R^\env)^n$ containing $\ker\mu$, and the image $\mu(M')$ of a left $R^\env$-submodule $M'$ of $(R^\env)^n$ with $\ker\mu\subseteq M'$ is an $R$-sub-bimodule of $R^n$.
The kernel of $\mu$ is generated by
$$(f_1\otimes 1,\dots,f_n\otimes 1)-(1\otimes f_1,\dots,1\otimes f_n)
\qquad (f_1,\dots,f_n\in R).$$

\subsection{The enveloping algebra of an algebra of solvable type}\label{The enveloping algebra of an algebra of solvable type}
{\it In the rest of this section, we assume that $R=K\<x\>$ is of solvable type with respect to $x=(x_1,\dots,x_N)$ and some monomial ordering $\leq$ of $\N^N$.
We let $\cal R=(R_{ij})$ be a commutation system defining $R$, with $R_{ij}$ as in \eqref{Rij}, and set $p_{ij}:=\pi(P_{ij})$, where $\pi\colon K\<X\>\to R$ is the natural surjection. }
The opposite $K$-algebra $R^\op$ of $R$ is again a $K$-algebra of solvable type in a natural way. To see this define the ``write oppositely automorphism'' of $K\<X\>$ by $$(X_{i_1}\cdots X_{i_r})^\op=X_{i_r}\cdots X_{i_1}\qquad\text{for all $i_1,\dots,i_r\in\N$.}$$ Also set $\alpha^\op:=(\alpha_N,\dots,\alpha_1)$ for every multi-index $\alpha=(\alpha_1,\dots,\alpha_N)$ and define the ``opposite ordering'' of $\N^N$ by
$$\alpha\leq^\op\beta \quad :\Longleftrightarrow\quad \alpha^\op\leq\beta^\op \qquad\text{for all multi-indices $\alpha$, $\beta$.}$$
Then $\cal R^\op:=(R_{ij}^\op)$ is a commutation system defining a $K$-algebra of solvable type with respect to $\leq^\op$ and $x^\op:=(x_N,\dots,x_1)$, which can be naturally identified with $R^\op$.

The class of $K$-algebras of solvable type is closed under tensor products. More precisely, let $\leq'$ be a monomial ordering of $\N^{N'}$ (where $N'\in\N$), and let  $\cal R'=(R'_{ij})$ be a commutation system in $K\<Y\>=K\<Y_1,\dots,Y_{N'}\>$, with
$$R'_{ij} = Y_jY_i - c'_{ij} Y_i Y_j - P'_{ij} \qquad (1\leq i<j\leq N')$$
where $0\neq c'_{ij}\in K$ and $P'_{ij}\in\bigoplus_{\alpha'} KY^{\alpha'}$.
(Here and below, $\alpha'$ ranges over $\N^{N'}$.) Let $R'=K\<Y\>/I(\cal R')$, with natural surjection $\pi'\colon K\<Y\>\to R'$, and let $y_j:=\pi'(Y_j)$ for  and $p_{ij}':=\pi'(P_{ij}')$. Suppose that $R'$ is of solvable type with respect to $\leq'$ and $y=(y_1,\dots,y_{N'})$. The $K$-algebra $S:=R\otimes_K R'$ is generated by the $(N+N')$-tuple
\begin{equation}\label{N+N'}
(x_1\otimes 1,\dots,x_N\otimes 1,1\otimes y_1,\dots,1\otimes y_{N'}).
\end{equation}
We have the following (see \cite[Proposition~1]{Roman-Roman}):

\begin{prop}
The $K$-algebra $S=R\otimes_K R'$ is of solvable type with respect to the lexicographic product of the orderings $\leq$ and $\leq'$, and the $(N+N')$-tuple of generators \eqref{N+N'}. The commutator relations of $S$ are
\begin{align*}
(x_j\otimes 1)(x_i\otimes 1) &= c_{ij} (x_i\otimes 1)(x_j\otimes 1) + p_{ij}\otimes 1 \qquad (1\leq i<j\leq N) \\
(x_i\otimes 1)(1\otimes y_j) &= (1\otimes y_j)(x_i\otimes 1) \qquad (1\leq i\leq N,\ 1\leq j\leq N')\\
(1\otimes y_j)(1\otimes y_i) &= c'_{ij} (1\otimes y_i)(1\otimes y_j) + 1\otimes p'_{ij} \qquad (1\leq i<j\leq N').
\end{align*}
Hence if $R$ and $R'$ are quadric, then so is $S$.
\end{prop}

In particular, $R^\env=R\otimes_K R^\op$ is an algebra of solvable type in a natural way, with respect to the monomial ordering $\leq^\env$ on $\N^{2N}=\N^N\times\N^N$ obtained by taking the lexicographic product of $\leq$ with $\leq^\op$. For every given $n$, the kernel of the left $R^\env$-morphism $\mu\colon (R^\env)^n\to R^n$ introduced in Section~\ref{The opposite algebra and the enveloping algebra} is generated by the elements
\begin{equation}\label{kernel of mu}
\big((x^{\varepsilon_i}\otimes 1)-(1\otimes x^{\varepsilon_i})\big)e_j \qquad (1\leq i\leq N,\ 1\leq j\leq n)
\end{equation}
of $(R^\env)^n$. Here $$\varepsilon_1=(1,0,\dots,0),\varepsilon_2=(0,1,0,\dots,0),\dots,\varepsilon_N=(0,\dots,0,1)\in\N^N,$$ and $e_1,\dots,e_n$ are the standard basis elements of the left $R^\env$-module $(R^\env)^n$.
Hence if $M$ is an $R$-sub-bimodule of $R^n$ generated by
$$f_i=(f_{i1},\dots,f_{in})\in R^n \qquad (i=1,\dots,m),$$
then the corresponding left $R^\env$-submodule $\mu^{-1}(M)$ of $(R^\env)^n$ is generated by the elements in \eqref{kernel of mu} and
$$(f_{11}\otimes 1,\dots,f_{1n}\otimes 1),\dots,(f_{m1}\otimes 1,\dots,f_{mn}\otimes 1).$$

\begin{cor}
Suppose $\leq$ is degree-compatible.
Let $f_1,\dots,f_n\in R$ be of degree at most $d$, and let $f\in R$. If there are a finite index set $J$ and $y_{ij},z_{ij}\in R$ \textup{(}$i=1,\dots, n$, $j\in J$\textup{)} such that
$$f = \sum_{j\in J} y_{1j} f_1 z_{1j} + \cdots + \sum_{j\in J} y_{nj} f_n z_{nj}$$
then there are such $J$ and $y_{ij}$, $z_{ij}$ with
$$\deg(y_{ij}),\deg(z_{ij})\leq \deg(f)+D(2N,d)\qquad\text{for $i=1,\dots,n$.}$$
\end{cor}
\begin{proof}
Apply Proposition~\ref{Degree-compatible}  to $R^\env$ and $$f_1\otimes 1,\dots,f_n\otimes 1, x^{\varepsilon_1}\otimes 1-1\otimes x^{\varepsilon_1},\dots,x^{\varepsilon_N}\otimes 1-1\otimes x^{\varepsilon_N}$$ in place of $R$ and $f_1,\dots,f_n$, respectively.
\end{proof}

The following observation (also from \cite{Roman-Roman}) allows one to compute  two-sided Gr\"obner bases in $R$ by computing one-sided Gr\"obner bases in the enveloping algebra of $R$:

\begin{prop} \label{Two-sided bases from one-sided}
Let $J$ be a two-sided ideal of $R$, and let $G$ be a Gr\"obner basis of the left ideal $\mu^{-1}(J)$ of $R^\env$. Then $\mu(G)$ is a Gr\"obner basis of $J$.
\end{prop}

So finally we can show:

\begin{proof}[Proof of Proposition~\ref{Two-sided}]
We may assume that $d>0$.
Suppose $J$ is a two-sided ideal of $R$ generated by $f_1,\dots,f_n\in R$ of degree at most $d$.
Let $\mu\colon R^\env\to R$ be as in Section~\ref{The opposite algebra and the enveloping algebra}.
The left ideal $\mu^{-1}(J)$ of $R^\env$ is generated by the elements
$$f_1\otimes 1,\dots,f_n\otimes 1, x^{\varepsilon_1}\otimes 1-1\otimes x^{\varepsilon_1},\dots,x^{\varepsilon_N}\otimes 1-1\otimes x^{\varepsilon_N},$$
each of which has degree at most $d$. By Corollary~\ref{Detailed bound}, $\mu^{-1}(J)$ has a Gr\"obner basis $G$ (with respect to $\leq^\env$) consisting of elements  of degree at most $D(2N,d)$. By Proposition~\ref{Two-sided bases from one-sided}, $\mu(G)$ is a Gr\"obner basis of $J$ whose elements obey the same degree bound.
\end{proof}

\bibliographystyle{amsplain} 

\providecommand{\bysame}{\leavevmode\hbox to3em{\hrulefill}\thinspace}

\end{document}